\documentclass{amsart}

\usepackage{amssymb}
\usepackage{array}

\theoremstyle{plain}

\newtheorem{theorem}{Theorem}[section]
\newtheorem{lemma}[theorem]{Lemma}

\newtheorem{corollary}[theorem]{Corollary}
\newtheorem{conjecture}[theorem]{Conjecture}
\newtheorem{definition}[theorem]{Definition}
\newtheorem{proposition}[theorem]{Proposition}

\newcommand{\makeinvisible}[1]{}

\numberwithin{equation}{section}

\newcommand{\cc}{{\mathbb C}}
\newcommand{\pp}{{\mathbb P}}

\newcommand{\aaa}{{\mathbb A}}

\newcommand{\Gg}{{\mathcal G}}

\newcommand{\Ff}{{\mathcal F}}
\newcommand{\Oo}{{\mathcal O}}

\newcommand{\Rr}{{\mathcal R}}

\newcommand{\Bb}{{\mathcal B}}

\begin{document}

\author[N. Mestrano]{Nicole Mestrano}
\address{Laboratoire J. A. Dieudonn\'e, CNRS UMR 7351
\\ Universit\'e C\^ote d'Azur\\
06108 Nice, Cedex 2, France}

\author[C. Simpson]{Carlos Simpson}
\address{Laboratoire J. A. Dieudonn\'e, CNRS UMR 7351
\\ Universit\'e C\^ote d'Azur\\
06108 Nice, Cedex 2, France}

\thanks{This research project was initiated on our visit to Japan supported by JSPS Grant-in-Aid for Scientific Research (S-19104002). This research
was supported in part by French ANR grants G-FIB (BLAN-08-3-352054),
SEDIGA (BLAN-08-1-309225), HODAG (09-BLAN-0151-02) and TOFIGROU
(933R03/13ANR002SRAR). We thank the University of Miami for hospitality during the completion of this work.}

\title[Irreducibility]{Irreducibility of the moduli space
of stable vector bundles of rank two and odd degree on a very general quintic surface}



\begin{abstract}
The moduli space $M(c_2)$, of stable rank two vector bundles of degree one
on a very general quintic surface $X\subset \pp^3$, is irreducible
for all $c_2\geq 4$ and empty otherwise.
\end{abstract}

\maketitle

\section{Introduction}
\label{introduction}

Let $X\subset \pp ^3_{\cc}$ be a very general quintic hypersurface.
Let $M(c_2):= M_X(2,1,c_2)$ denote the moduli space \cite{HuybrechtsLehn}
of stable rank $2$ vector bundles on $X$ of degree $1$ with $c_2(E)=c_2$.
Let $\overline{M} (c_2):= \overline{M}_X(2,1,c_2)$ denote the
moduli space of stable rank $2$ torsion-free sheaves on $X$ of degree $1$ with $c_2(E)=c_2$.
Recall that $\overline{M} (c_2)$ is projective, and
$M(c_2)\subset \overline{M} (c_2)$ is an open set, whose complement is called the {\em boundary}.
Let $\overline{M(c_2)}$ denote the closure of $M(c_2)$ inside $\overline{M} (c_2)$.
This might be a strict inclusion, as will in fact be the case for $c_2\leq 10$.

In \cite{MS1} we showed that $M(c_2)$ is irreducible for $4\leq c_2\leq 9$, and empty for $c_2\leq 3$.
In \cite{MS2}
we showed that the open subset $M(10)^{\rm sn}\subset M(10)$, of bundles with seminatural cohomology, is irreducible. In 1995 Nijsse
\cite{Nijsse} showed that $M(c_2)$ is irreducible for $c_2\geq 16$.

In the present paper, we complete the proof of irreducibility for the remaining intermediate values of $c_2$.

\begin{theorem}
\label{maintheorem}
For any $c_2\geq 4$, the moduli space of bundles $M(c_2)$ is irreducible.

For $c_2\geq 11$, the moduli space of
torsion-free sheaves $\overline{M} (c_2)$ is irreducible. On the other hand, $\overline{M} (10)$ has two irreducible components:
the closure $\overline{M(10)}$ of the irreducible open set $M(10)$;
and the smallest stratum $M(10,4)$ of the double dual stratification
corresponding to torsion-free sheaves whose double dual has $c'_2=4$.
Similarly $\overline{M}(c_2)$ has several irreducible components when $5\leq c_2\leq 9$ too.

The moduli space $\overline{M} (c_2)$ is good for $c_2\geq 10$,
generically smooth of the expected dimension $4c_2-20$, whereas
for $4\leq c_2\leq 9$, the moduli space $M(c_2)$ is not good.
For $c_2\leq 3$ it is empty.
\end{theorem}

Yoshioka \cite{YoshiokaAppli, YoshiokaK3, YoshiokaAbelian}, Gomez \cite{GomezThesis} and others have shown that
the moduli space of stable torsion-free sheaves with irreducible Mukai vector (which contains, in particular, the case
of bundles of rank $2$ and degree $1$) is irreducible, over an abelian or K3 surface. Those results use
the triviality of the canonical bundle, leading to a symplectic structure and implying among other things that the
moduli spaces are smooth \cite{Mukai}. Notice that the case of K3 surfaces includes degree $4$ hypersurfaces in $\pp^3$.

We were motivated to look at a next case, of bundles on a quintic or degree $5$ hypersurface in $\pp ^3$ where $K_X=\Oo _X(1)$
is ample but not by very much. This paper is the third in a series starting with \cite{MS1, MS2}
dedicated to Professor Maruyama who, along with Gieseker, pioneered the
study of moduli of bundles on higher dimensional varieties \cite{Gieseker, GiesekerCons, Maruyama, MaruyamaET, MaruyamaTransform}.
Recall that the moduli space of stable bundles is irreducible for $c_2\gg 0$ on any smooth projective surface \cite{GiesekerLi, Li, OGradyIrred, OGradyBasic},
but there exist surfaces, such as smooth hypersurfaces in $\pp ^3$ of sufficiently high degree  \cite{Mestrano},
where the moduli space
is not irreducible for intermediate values of $c_2$.

Our theorem shows that
the irreducibility of the moduli space of bundles $M(c_2)$, for all values of $c_2$,
can persist into the range where $K_X$ is ample. On the other hand, the fact that $\overline{M} (10)$
has two irreducible components, means that if we consider all torsion-free sheaves, then the property of irreducibility
in the good range has
already started to fail in the case of a quintic hypersurface.
We furthermore show in Section \ref{example6} below that irreducibility fails for stable vector bundles on surfaces of degree $6$.

A possible application of our theorem to the case of Calabi-Yau varieties could be envisioned, by noting that a general hyperplane section of
a quintic threefold in $\pp ^4$ will be a quintic surface $X\subset \pp^3$.

\medskip

\begin{center}
{\large \em Outline of the proof}
\end{center}

Our technique is to use O'Grady's method of deformation to the boundary
\cite{OGradyIrred, OGradyBasic}, as it was exploited by Nijsse \cite{Nijsse}
in the case of a very general quintic hypersurface. We use, in particular, some of the intermediate
results of Nijsse who showed, for example, that $\overline{M} (c_2)$ is connected for $c_2\geq 10$.
Application of these results is made possible by the explicit description of the moduli spaces
$M(c_2)$ for $4\leq c_2\leq 9$ obtained in \cite{MS1} and the partial result for $M(10)$ obtained
in \cite{MS2}.

The boundary $\partial \overline{M} (c_2):= \overline{M} (c_2)-M(c_2)$ is the set of
points corresponding to torsion-free sheaves which are not locally free. We
just endow $\partial \overline{M} (c_2)$ with its reduced scheme structure. There might in
some cases be a better non-reduced structure which one could put on the boundary or onto some strata,
but that won't be necessary for our argument and we don't worry about it here.

We can further refine the decomposition
$$
\overline{M} (c_2)=M(c_2)\sqcup \partial \overline{M} (c_2)
$$
by the {\em double dual stratification} \cite{OGradyBasic}.
Let $M (c_2; c_2')$ denote the
locally closed subset, again with its reduced scheme structure,
parametrizing sheaves $F$ which fit into an exact sequence
$$
0\rightarrow F \rightarrow F^{\ast\ast} \rightarrow S \rightarrow 0
$$
such that $F\in \overline{M} (c_2)$ and
$S$ is a coherent sheaf of finite length $d=c_2-c_2'$ hence $c_2(F^{\ast\ast})=c_2'$. Notice that
$E=F^{\ast\ast}$ is also stable so it is a point in $M(c'_2)$. The stratum can be nonempty
only when $c'_2\geq 4$, which shows by the way that $\overline{M}(c_2)$ is empty for $c_2\leq 3$.
The boundary now decomposes into
locally closed subsets
$$
\partial \overline{M} (c_2) = \coprod _{4 \leq c_2' < c_2} M (c_2;c'_2).
$$
Let $\overline{M (c_2,c_2')}$ denote the closure of $M(c_2,c_2')$ in $\overline{M} (c_2)$. Notice that
we don't know anything about the position of this closure with respect to the stratification;
its boundary will not in general be a union of strata. We can similarly denote by $\overline{M(c_2)}$ the
closure of $M(c_2)$ inside $\overline{M} (c_2)$, a subset which might well be strictly smaller than $\overline{M} (c_2)$.

The construction $F\mapsto F^{\ast\ast}$ provides, by the definition of the stratification,
a well-defined map
$$
M (c_2;c'_2)\rightarrow M(c'_2).
$$
The fiber over $E\in M(c'_2)$ is the Grothendieck ${\rm Quot}$-scheme ${\rm Quot}(E;d)$ of quotients of $E$ of length $d:=c_2-c'_2$.

It follows from Li's theorem \cite[Proposition 6.4]{Li} that
if $M(c_2')$ is irreducible, then $M(c_2;c'_2)$ and hence $\overline{M(c_2;c'_2)}$ are irreducible,
with ${\rm dim} (M(c_2;c'_2) )= {\rm dim} ( M(c_2') )+ 3 (c_2-c_2')$.  See Corollary \ref{forward} below.
From the previous papers \cite{MS1, MS2},
we know the dimensions of $M(c_2')$, so we can fill in the dimensions of the strata, as will be
summarized in Table \ref{maintable}. Furthermore, by \cite{MS1} and Li's theorem, the strata
$M(c_2;c'_2)$ are irreducible whenever $c'_2\leq 9$.

Nijsse \cite{Nijsse} proves that $\overline{M} (c_2)$ is connected whenever $c_2\geq 10$, using O'Grady's techniques
\cite{OGradyIrred, OGradyBasic}. This is discussed in \cite{MS3}.
By \cite{MS1}, the moduli space $\overline{M} (c_2)$ is {\em good},
that is to say it is generically reduced of the expected dimension $4c_2-20$, whenever $c_2\geq 10$. In particular,
the dimension of the Zariski tangent space, minus the dimension of the space of obstructions, is equal to the
dimension of the moduli space. The Kuranishi theory of deformation spaces implies that $\overline{M} (c_2)$
is locally a complete intersection. Hartshorne's connectedness theorem \cite{HartshorneConnectedness}
now says that if two different irreducible components of $\overline{M} (c_2)$ meet at some point, then they intersect
in a codimension $1$ subvariety. This intersection has to be contained in the singular locus.

The singular locus in $M(c_2)$ contains a subvariety denoted $V(c_2)$, which is the set of bundles $E$ with $h^0(E)>0$.
It is the image of the space $\Sigma _{c_2}$ of
extensions
$$
0\rightarrow \Oo _X\rightarrow E\rightarrow J_P(1)\rightarrow 0
$$
where $P$ satisfies Cayley-Bacharach for quadrics.
For $c_2\geq 10$, $V(c_2)$ is irreducible of dimension $3c_2-11$.
For $c_2\geq 11$ one can see directly that
the closure of $V(c_2)$
meets the boundary. For $c_2=10$, bundles in $V(10)$
almost have seminatural
cohomology, in the sense that any deformation moving
away from $V(10)$ will have seminatural cohomology,
so $V(10)$ is contained only
in the irreducible component constructed in \cite{MS2}, and that component meets the boundary.
On the other hand, any other irreducible components of the singular locus have
strictly smaller dimension \cite[Corollary 7.1]{MS1}.

These properties of the singular locus, together with the connectedness statement of \cite{Nijsse}, allow us to
show that any irreducible component of $\overline{M} (c_2)$ meets the boundary. O'Grady proves furthermore an important
lemma, that the
intersection with the boundary must have pure codimension $1$.

We explain the strategy for proving irreducibility of $M(10)$ and $M(11)$ below, but it will perhaps be
easiest to explain first why this implies irreducibility of $M(c_2)$ for $c_2\geq 12$. Based on O'Grady's method,
this is the same strategy as was used by Nijsse who treated the cases $c_2\geq 16$.

Suppose $c_2\geq 12$ and $Z\subset \overline{M} (c_2)$ is an irreducible component.
Suppose inductively we know that $M(c_2-1)$ is irreducible.
Then $\partial Z:= Z\cap \partial \overline{M} (c_2)$
is a nonempty subset in $Z$ of codimension $1$, thus of dimension
$4c_2-21$.  However, by looking at Table \ref{maintable}, the boundary $\partial \overline{M} (c_2)$
is a union of the stratum $M(c_2,c_2-1)$ of dimension $4c_2-21$, plus other strata of strictly smaller dimension.
Therefore, $\partial Z$ must contain $M(c_2,c_2-1)$. But, the general torsion-free sheaf parametrized by a point of
$M(c_2,c_2-1)$ is the kernel $F$ of a general surjection $E\rightarrow S$ from a stable bundle $E$ general in $M(c_2-1)$,
to a sheaf $S$ of length $1$. We claim that $F$ is a smooth point of the moduli space $\overline{M} (c_2)$.
Indeed, if $F$ were a singular point then there would exist a nontrivial co-obstruction $\phi : F\rightarrow F(1)$,
see \cite{Langer, MS1, Zuo}. This would have to come from
a nontrivial co-obstruction $E\rightarrow E(1)$ for $E$, but that cannot exist because a general $E$ is a smooth point since $M(c_2-1)$
is good. Thus, $F$ is a smooth point of the moduli space. It follows that a given irreducible component of
$M(c_2,c_2-1)$ is contained in at most one irreducible component of $\overline{M} (c_2)$. On the other hand,
by the induction hypothesis $M(c_2-1)$ is irreducible, so $M(c_2,c_2-1)$ is irreducible. This gives the
induction step, that $M(c_2)$ is irreducible.

The strategy for $M(10)$ is similar. However, due to the fact that the moduli spaces $M(c'_2)$ are
not good for $c'_2\leq 9$, in particular they tend to have dimensions bigger than the expected dimensions,
there are several boundary strata which can come into play. Luckily, we know that the $M(c'_2)$,
hence all of the strata $M(10,c'_2)$, are irreducible for $c'_2\leq 9$.

The dimension of $M(10)$,
equal to the expected one, is $20$. Looking at the row $c_2=10$
in Table \ref{maintable} below, one may see that there are three strata $M(10,9)$, $M(10,8)$ and $M(10,6)$
with dimension $19$. These can be irreducible components of the  boundary $\partial Z$ if we follow the
previous argument. More difficult is the case of the stratum $M(10,4)$ which has dimension $20$.
A general point of $M(10,4)$ is not in the closure of $M(10)$, in other words
$M(10,4)$, which is closed since it is the lowest stratum, constitutes a separate
irreducible component of $\overline{M} (10)$. Now, if $Z\subset M(10)$ is an irreducible component,
$\partial Z$ could contain a codimension $1$ subvariety of $M(10,4)$.

The idea is to use the main result of \cite{MS2}, that the moduli space $M(10)^{\rm sn}$
of bundles with seminatural cohomology, is irreducible. To prove that $M(10)$ is irreducible,
it therefore suffices to show that a general point of any irreducible component $Z$, has seminatural
cohomology. From \cite{MS2} there are two conditions that need to be checked: $h^0(E)=0$
and $h^1(E(1))=0$. The first condition is automatic for a general point, since the locus
$V(10)$ of bundles with $h^0(E)>0$ has dimension $3\cdot 10-11=19$ so cannot contain a general point of $Z$.
For the second condition, it suffices to note that a general sheaf $F$ in any of the strata $M(10,9)$, $M(10,8)$
and $M(10,6)$ has $h^1(F(1))=0$; and to show that the subspace of sheaves $F$ in $M(10,4)$ with
$h^1(F(1))>0$ has codimension $\geq 2$. This latter result is treated in Section \ref{lowest}, using
the dimension results of Ellingsrud-Lehn for the scheme of quotients of a locally free sheaf,
generalizing Li's theorem. This is how we will show irreducibility of $M(10)$.

The full moduli space of torsion-free sheaves $\overline{M} (10)$ has two different irreducible
components, the closure $\overline{M(10)}$ and the lowest stratum $M(10,4)$. This
distinguishes the case of the quintic surface from the cases of abelian and K3 surfaces,
where the full moduli spaces of stable torsion-free sheaves were irreducible \cite{YoshiokaAbelian,
YoshiokaK3, GomezThesis}.

For $M(11)$, the argument is almost the same as for $c_2\geq 12$. However, there are now two different
strata of codimension $1$ in the boundary: $M(11,10)$ coming from the irreducible variety $M(10)$,
and $M(11,4)$ which comes from the other $20$-dimensional component $M(10,4)$ of $\overline{M} (10)$.
To show that these two can give rise to at most a single irreducible component in $M(11)$, completing the proof,
we will note that they do indeed intersect, and furthermore that the intersection contains smooth points.

After the end of the proof of Theorem \ref{maintheorem}, the last two sections
of the paper
treat some related considerations.

In Section \ref{correx} we provide a correction and improvement to \cite[Lemma 5.1]{MS1} and
answer \cite[Question 5.1]{MS1}. Recall from there that a co-obstruction may be interpreted as a sort of
Higgs field with values in the canonical bundle $K_X$; it has a spectral surface $Z\subset {\rm Tot}(K_X)$.
The question was to bound the irregularity of a resolution of singularities of the spectral surface $Z$. We show in
Lemma \ref{51bis} that the irregularity vanishes.

At the end of the paper in Section \ref{example6},
we show that Theorem \ref{maintheorem} is sharp  as far as the degree $5$ of the very general hypersurface is concerned.
In the case of bundles on very general hypersurfaces $X^6$ of degree $6$, we show in
Theorem \ref{twocomponents} that the
moduli space $M_{X^6}(2,1,11)$ of stable rank two bundles of degree $1$ and $c_2=11$ has at least two irreducible components.
This improves the result of \cite{Mestrano}, bringing from $27$ down to $6$ the degree of a very general hypersurface on
which there exist two irreducible components. We expect that there will be several irreducible
components in any degree $\geq 6$ but that isn't shown here.

\section{Preliminary facts}

The moduli space  $\overline{M} (c_2)$ is locally a fine moduli space. The obstruction to
existence of a Poincar\'e universal sheaf on $\overline{M} (c_2)\times X$ is an interesting
question but not considered in the present paper. A universal family exists etale-locally
over $\overline{M} (c_2)$ so for local questions we may consider $\overline{M} (c_2)$ as a
fine moduli space.

The Zariski tangent space to $\overline{M} (c_2)$ at a point $E$ is ${\rm Ext}^1(E,E)$.
If $E$ is locally free, this is the same as $H^1({\rm End}(E))$. The {\em space of obstructions}
${\rm obs}(E)$ is by definition the kernel of the surjective map
$$
{\rm Tr}: {\rm Ext}^2(E,E)\rightarrow H^2(\Oo _X).
$$
The {\em space of co-obstructions} is the dual ${\rm obs}(E)^{\ast}$ which is, by Serre duality with $K_X=\Oo _X(1)$,
equal to ${\rm Hom}^0(E,E(1))$, the space of maps $\phi : E\rightarrow E(1)$ such that ${\rm Tr}(E)=0$ in
$H^0(\Oo _X(1))\cong \cc ^4$. Such a map is called a {\em co-obstruction}.

Since a torsion-free sheaf $E$ of rank two and odd degree can have no rank-one subsheaves of the same
slope, all semistable sheaves are stable, and Gieseker and slope stability are equivalent.
If $E$ is a stable sheaf then ${\rm Hom}(E,E)= \cc$ so the space of trace-free endomorphisms is
zero. Notice that $H^1(\Oo _X)=0$ so we may disregard the trace-free condition for
${\rm Ext}^1(E,E)$.  An Euler-characteristic calculation gives
$$
{\rm dim}(Ext^1(E,E)) - {\rm dim}({\rm obs}(E)) =4c_2-20,
$$
and this is called the {\em expected dimension} of the moduli space. The moduli space is said to be {\em good}
if the dimension is equal to the expected dimension.

\begin{lemma}
\label{lci}
If the moduli space is good, then it is locally a complete intersection.
\end{lemma}
\begin{proof}
Kuranishi theory expresses the local analytic germ of the
moduli space $\overline{M} (c_2)$ at $E$, as $\Phi ^{-1}(0)$ for a holomorphic map
of germs $\Phi : (\cc ^a ,0) \rightarrow (\cc ^b,0)$ where $a={\rm dim}(Ext^1(E,E))$ (resp. $b={\rm dim}({\rm obs}(E))$).
Hence, if the moduli space has dimension $a-b$, it is a local complete intersection.
\end{proof}

We investigated closely the structure of the moduli space for $c_2\leq 9$, in \cite{MS1}.

\begin{proposition}
\label{leq9}
The moduli space $M(c_2)$ is empty for $c_2\leq 3$. For $4\leq c_2\leq 9$, the moduli space
$M(c_2)$ is irreducible. It has dimension strictly bigger than the expected one, for $4\leq c_2\leq 8$,
and for $c_2=9$ it is generically nonreduced but with dimension equal to the expected one;
it is also generically nonreduced for $c_2=7$.
The dimensions of the moduli spaces, the dimensions of the spaces of obstructions at a general point,
and the dimensions $h^1(E(1))$
for a general bundle $E$ in $M(c_2)$, are given in the following table.
\begin{table}[h]
\caption{\label{leq9table} Moduli spaces for $c_2\leq 9$}
\begin{tabular}{|c|c|c|c|c|c|c|}
\hline
$c_2$   & $4$ &  $5$     & $6$   & $7$   & $8$   & $9$       \\
\hline
${\rm dim}(M)$  & $2$  & $3$     & $7$   & $9$   & $13$   & $16$      \\
\hline
${\rm dim}({\rm obs})$   & $6$  & $3$     & $3$   & $3$   & $1$   & $1$       \\
\hline
$h^1(E(1))$   & $0$  & $1$     & $0$  & $0$   & $0$  & $0$       \\
\hline
${\rm generically}$ & ${\rm sm}$ & ${\rm sm}$ & ${\rm sm}$
& ${\rm nr}$ & ${\rm sm}$ & ${\rm nr}$
\\
\hline
\end{tabular}
\end{table}
\end{proposition}

The proof of Proposition \ref{leq9} will be given in the next section, with a review of the cases $c_2\leq 9$ from the paper
\cite{MS1}.

We also proved that the moduli space is good for $c_2\geq 10$, known by Nijsse \cite{Nijsse} for $c_2\geq 13$.

\begin{proposition}
\label{singularlocus}
For $c_2\geq 10$, the moduli space $M(c_2)$ is good. The singular locus $M(c_2)^{\rm sing}$
is the union of the locus $V(c_2)$ consisting of bundles with $h^0(E)>0$, which has
dimension $3c_2-11$, plus other pieces of dimension $\leq 13$ which in particular have codimension $\geq 6$.
\end{proposition}
\begin{proof}
Following O'Grady's and Nijsse's terminology $V(c_2)$ denotes the locus which which is the image of
the moduli space of bundles together with a section, called $\Sigma _{c_2}$ or sometimes $\{ E,P\}$.
See \cite[Theorem 7.1]{MS1}.
Any pieces of the singular locus corresponding to bundles which are not
in $V(c_2)$, have dimension $\leq 13$ by \cite[Corollary 5.1]{MS1} (see Lemma \ref{51bis} below for a correction and
improvement of this statement).
\end{proof}

The case $c_2=10$ is an important central point in the classification, where the case-by-case treatment gives way to a general picture. In \cite{MS2} we proved the following partial result that will be used in the present paper to complete the proof of irreducibility.

\begin{proposition}
\label{eq10}
Let $M(10)^{\rm sn}\subset M(10)$ denote the open subset of bundles $E\in M(10)$ which have {\em seminatural cohomology},
that is where for any $m$ at most one of $h^i(E(m))$ is nonzero for $i=0,1,2$. Then $E\in M(10)^{\rm sn}$
if and only if $h^0(E)=0$ and $h^1(E(1))=0$. The moduli space $M(10)^{\rm sn}$
is irreducible.
\end{proposition}
\begin{proof}
See  \cite{MS2}, Theorem 0.2 and Corollary 3.5.
\end{proof}

\section{Review of $c_2\leq 9$}

The dicussion of the moduli spaces for $c_2\leq 9$
went by a sometimes exhaustive
classification of cases \cite[Lemmas 7.3, 7.4]{MS1}. In retrospect we can give more uniform proofs of some parts.
For this reason, and for the reader's convenience, it is worthwhile to review here
some of the arguments leading to the proof of Proposition \ref{leq9}.

There is a
change of notation with respect to \cite{MS1}. There
we considered bundles of degree $-1$. The bundle
of degree $1$ denoted here by $E$ is the same as the bundle
denoted by $E(1)$ in \cite{MS1}.
Thus \cite[Lemma 5.2]{MS1} speaks of $h^1(E)$ in our
notation.
The present notation was already in effect in \cite{MS2}.
Fortunately, the indexing by second Chern class
remains the same in both cases.

Following O'Grady, we
denote by $V(c_2)\subset M(c_2)$ the subvariety of bundles such that $h^0(E)>0$.
For $c_2\leq 9$ the Euler characteristic argument of
\cite[\S 6.1]{MS1} tells us that $h^0(E)>0$ for any $E$,
so $V(c_2)$ is the full moduli space.

It will be useful to consider the moduli space $\Sigma _{c_2}$ consisting of
pairs $(E,\eta )$ where $E\in M(c_2)$ and $\eta \in H^0(E)$ is a nonzero section.
The pairs are taken up to isomorphism, i.e. up to scaling of the section, so the
fiber of the map $\Sigma _{c_2}\rightarrow V(c_2)$ over a bundle $E$ is the
projective space $\pp H^0(E)$.

Each irreducible component of $\Sigma _{c_2}$
has dimension
$\geq 3c_2-11$, see \cite{OGradyBasic, Nijsse} or
\cite[Corollary 3.1]{MS1}.

A point of $\Sigma _{c_2}$ may also be considered as an
extension of the form
$$
0\rightarrow \Oo _X \rightarrow E \rightarrow J_{P/X}(1)\rightarrow 0,
$$
again up to isomorphism. We therefore employ the notation $\{ E,P\}:= \Sigma _{c_2}$ too.

Such an extension exists, with $E$ a bundle, if and only if
$P\subset X$ is locally a complete intersection
of length $c_2$ and satisfies the Cayley-Bacharach condition for quadrics denoted $CB(2)$.
See \cite{Barth, GriffithsHarris, Reider}
and the references for the Hartshorne-Serre correspondence
discussed in \cite{Arrondo} for the origins of this principle.

Denote by $\{ P\}$ the Hilbert scheme of l.c.i.\
subschemes $P$ that satisfy $CB(2)$.
The map $\{ E,P\}\rightarrow \{ P\}$
has fibers described as follows: the fiber over $P$ is
a dense open subset\footnote{It is the open subset of extensions such that $E$ is locally free, nonempty because of the conditions on $P$.}
of the projective space of all extensions
$\pp {\rm Ext}^1(J_{P/X}(1),\Oo _X)$; its dimension by duality is $h^1(J_{P/X}(1))-1$.

Consider $c$ the number of conditions imposed by $P$ on
quadrics. This is related to $h^1(E(1))$ by the exact sequences
$$
H^0(\Oo _X(2))\rightarrow H^0(\Oo _P(2))\rightarrow H^1(J_{P/X}(2)) \rightarrow 0
$$
and
$$
0\rightarrow H^1(E(1)) \rightarrow H^1(J_{P/X}(2))\rightarrow H^2(\Oo _X(1)) \rightarrow 0
$$
where $H^2(E(1))= H^0(E(1))^{\ast}=0$ by stability, and $H^2(\Oo _X(1))=H^2(K_X)=\cc$.
The number $c$ is the rank of the evaluation map of $H^0(\Oo _X(2))$
on $P$, so $h^1(J_{P/X}(2))= c_2-c$, and by the second exact sequence we have $h^1(E(1))=c_2-c-1$.

The number $c_2-c-1$ is also equal to the dimension of the fiber
of the map from the space of extensions $\{ E,P\}$
to the Hilbert scheme of subschemes $\{ P\}$.
As stated previously, the space of extensions $\{ E,P\}$ fibers
over the moduli space of bundles $\{ E\}$ with fiber
${\mathbb P}H^0(E)$ of dimension $h^1(J_{P/X}(1))$.

The locus $V(c_2)$, image of $\Sigma _{c_2}$, is the main piece of the set of potentially obstructed bundles, that is to say bundles for which the space of obstructions is nonzero.

The other pieces are of smaller dimension.
There was an error in the proof of this dimension estimate,
Lemma 5.1 and hence Corollary 5.1 in \cite{MS1}. These
will be corrected and improved in a separate section at the end of the present paper, see Lemma \ref{51bis} below.

\subsection{Using the Cayley-Bacharach condition}

Recall that a subscheme $P\subset \pp^3$ satisfies the {\em Cayley-Bacharach condition} $CB(n)$ if, for any subscheme $P'\subset P$ with $\ell (P')=\ell (P)-1$, a section $f\in H^0(\Oo _{\pp ^3}(n))$ vanishing on $P'$ must also
vanish on $P$.
When $P\subset X$ this is the condition governing the existence of an extension of $J_{P/X}(n-1)$ by $\Oo _X$ that is locally free. For the study of $\Sigma _{c_2}$ we are therefore interested in subschemes satisfying $CB(2)$.

See \cite{MS1}, \cite{MS2} and the survey \cite{MS3} for
details on the basic techniques we use to analyse the Cayley-Bacharach condition.

If $U\subset \pp^3$ is a divisor, usually for us a plane, and $P$ a subscheme, there is a {\em residual subscheme} $P'$
for $P$ with respect to $U$. In the case of distinct points it is just the complement of $P\cap U$, but more generally it has a schematic meaning with $\ell (P')+\ell (P\cap U)=\ell (P)$.
If $P$ satisfies $CB(n)$ and $U$ has degree $m$ then the residual $P'$ satisfies $CB(n-m)$.

The following fact will be used often: if $P'$ is the residual of $P$ with respect to $U$, and if $Z\subset \pp^3$ is a subvariety, then the length of $Z\cap P$ at any point is
at least equal to the length of $Z\cap P'$. So for example if $P'$ has $3$ points in a line (schematically),
then $P$ does too.

It is easy to see that the Cayley-Bacharach condition $CB(2)$ cannot be satisfied by $\leq 3$ points, so the moduli space is empty for $c_2\leq 3$.
Here is a case-by-case review of the cases $4\leq c_2\leq 9$.

\subsection{For $c_2=4,5$}

Here the subscheme $P$ is either $4$ or $5$ points contained in a line. Both of these
configurations impose $c=3$ conditions on quadrics, since $h^0(\Oo _{\pp ^1}(2))=3$. This gives
values of $4-3-1=0$ and $5-3-1=1$ for $h^1(E(1))$ respectively. The moduli space is generically smooth and its dimension
is equal to $c_2-2$ by \cite[Lemma 7.7]{MS1}. This may be seen directly from the more explicit descriptions we shall give in Section \ref{lowest} below.
We get the dimension of the space of
co-obstructions by subtracting the expected dimension.
This completes the proof of Proposition \ref{leq9} for the columns $c_2=4,5$.

\subsection{For $c_2=6,7$}

In both cases, the Euler characteristic argument
of \cite[Section 6.1]{MS1} gives $h^0(E)=2$,
hence $h^0(J_{P/X}(1))=1$ and $P$ is contained in a
unique plane $U$.
By \cite[Lemma 5.5]{MS1}, the space of obstructions has dimension $3$.

For $c_2=\ell (P) =6$, see \cite[Proposition 7.4]{MS1} that we now review.
The number $c$ of conditions
imposed on quadrics has to be $\leq 5$, in particular $P$ is contained in a planar conic $Y\subset U$.
However, $c\leq 4$ may be ruled out by the size of $P$ and the Cayley-Bacharach condition, see the second paragraph of \cite[\S 7.5]{MS1}. It follows that the dimension of $\{ E,P\}$ equals the dimension of $\{ P\}$, and as noted above this dimension is $\geq 3c_2-11 = 7$.

Look at the family of length $6$ subschemes $P\subset
X\cap Y$ such that all points of $P$ are located either
at smooth points of $Y$, or at smooth points of $X\cap U$.
Such a subscheme is uniquely determined by its multiplicities
at each point, so given $Y$ the set of choices of $P$ is
discrete and if we generalize $Y$, the subscheme $P$ generalizes. Therefore, this defines a set of irreducible components of dimension is equal to the dimension of the space of choices of $Y$, that is $8$. For $U$ fixed and $Y$ general,
the choice of $P$ is equivalent to the choice of complementary set of $4$ points in $Y\cap X$; but since any $4$ points
in the plane lie on a conic, the monodromy action as we  move $Y$ can take any choice of $4$ points to any other one. Therefore,
this family is a single irreducible component of dimension $8$.

The remaining locus of $P$ containing a point where
$Y$ is singular and
$U$ is tangent to $X$, has dimension $\leq 5$. For example if
there is one such point, then the space of choices of $U$ has
dimension $2$; the space of choices of $Y$ has dimension
$2$; and by the precise estimate of
\cite[Proposition 4.3]{BrianconGrangerSpeder}, noting that
$Y$ has multiplicity $2$ at the singular point, the space of choices of $P$ has dimension $\leq 1$. For more points, we get one further dimension of the space of choices of $P$ for each other point but more than $1$ new condition imposed by the tangencies. Therefore, the locus of subschemes not fitting into the situation of the previous paragraph, has dimension $<7$ and it cannot produce a new irreducible component.

This completes the discussion for $c_2=6$: we have an irreducible component of $\{ E,P\}$ of dimension $8$ whose general point consists of a choice of $6$ out of the $10$ intersection points in $X\cap Y$ for a plane conic $Y$.
Since $h^0(E)=2$ the dimension of $\{ E\}$ is $7$.
For the table,
notice that $h^1(E(1))=6-5-1=0$. Comparing dimension, expected dimension $4\cdot 6-20=4$ and the dimension $3$ of the space of obstructions, we find that the moduli space is generically smooth with vanishing obstruction maps.

Consider now the case $c_2=7$. See \cite[Proposition 7.3]{MS1}
to be reviewed as follows.
As previously  from the second paragraph of \cite[\S 7.5]{MS1}, the case $c\leq 4$ may be ruled out.
If $c=5$, then $P$ would be contained in a plane
conic $Y\subset U$, but using the same arguments as before
the dimension of the space of choices of $P$ would be $\leq 8$;
however any irreducible component of $\{ E,P\}$
has dimension $\geq 3\cdot 7 -11 =10$ and the fiber of the map to $\{ P\}$ has dimension $1$, so a family of subschemes $P$ of dimension $\leq 8$ cannot contribute an irreducible component. Therefore we may
suppose $c=6$, the dimensions of $\{ E,P\}$ and
$\{ P\}$ are the same and are $\geq 10$.
For a given plane $U$ the space of choices
of subscheme $P\subset X\cap U$ of length $7$, has dimension $7$ by \cite{BrianconGrangerSpeder}.
The space of choices of $P$ such that $U\cap X$  is singular
(i.e. $U$ tangent to $X$), therefore has dimension $\leq 9$ and cannot contribute.
If $U$ is a plane such that $X\cap U$ is smooth, the
Hilbert scheme of $P\subset X\cap U$ is irreducible and a general point corresponds to choosing $7$ distinct points.
We conclude that $\{ E,P\}$ is
irreducible of dimension $10$ with general point consisting of
a general subscheme $P\subset
U\cap X$ of length $7$ that indeed satisfies $CB(2)$ imposing $c=6$ conditions on quadrics.

Notice that since $h^0(E)=2$ the map $\{ E,P\} \rightarrow \{ E\}$ is a fibration with fibers $\pp^1$ so the corresponding irreducible component of the moduli space has dimension $9$
as filled into the table. At a general point
where $P$ imposes $c=6$ conditions on quadrics, we get $h^1(E(1))=7-6-1=0$. From \cite[Proposition 7.3]{MS1},
by comparing dimensions
the moduli space is generically nonreduced.
This treats the column $c_2=7$.

\subsection{For $c_2=8$}
See the discussion in
\cite[Section 6.2]{MS1} and \cite[Theorem 7.2]{MS1} which will now be reviewed with some improvement in the arguments allowing us to bypass certain case-by-case considerations.

Any component of $\{ E,P\}$ has dimension $\geq
3\cdot 8 -11 = 13$.

The following technique, involving the residual subscheme recalled above, will be useful.

\begin{lemma}
\label{uuprime}
Suppose $U\subset \pp^3$ is a plane, and let $P'$ denote the residual subscheme for
$P$ with respect to $U$. If nonempty $P'$ satisfies $CB(1)$, so $\ell (P')\geq 3$ and
in case of equality $P'$ is colinear.

Let $n$ be the number of additional conditions needed to insure vanishing on $U$ of quadrics passing through $P$.
Suppose $10-c \geq n+1$. Then there exists a quadric containing $P$ of the form $U.U'$ where $U'$ is another plane, containing $P$.
In particular, $P'\subset U'$. If $10-c \geq n+2$ then
$P'$ is contained in a line, and if $10-c\geq n+3$ then $P\subset U$.
\end{lemma}
\begin{proof}
The first paragraph is a restatement of the basic property of the residual subscheme. Note that one or two points, or three non-colinear points, cannot be $CB(1)$.

In the second paragraph, we could define $n$ as the
dimension of the image of
$$
H^0(J_{P/\pp^3}(2))\rightarrow H^0(\Oo _U(2)).
$$
If  $10-c \geq n+1$ then it means that we can  impose $n$ additional conditions (say, vanishing at general points of $U$) on the
$(10-c)$-dimensional space quadrics $H^0(J_{P/\pp^3}(2))$, to get one that vanishes on $U$. This quadric has the form $U.U'$
of the union of $U$ with another plane $U'$. By definition the residual is contained in $U'$. If $10-c \geq n+2$ then the $U'$ move in a $2$-dimensional family so they cut out a line containing $P'$. If $10-c \geq n+3$ the family of $U'$ cuts out a point, however $P'$ satisfying $CB(1)$ cannot be a single point so in this case it is empty and $P\subset U$.
\end{proof}

Look at the value of $c$ at a general point of an irreducible
component. The case $c\leq 5$ may be ruled out
(using a simpler version of the subsequent arguments),
so we may assume either $c=6$ or $c=7$. If $c=6$ then
the fiber of $\{ E,P\}\rightarrow \{ P\}$
has dimension $1$ and $\{ P\}$ has dimension $\geq 12$,
whereas if $c=7$ then the irreducible component of
$\{ E,P\}$ is the same as that of $\{ P\}$, and $\{ P\}$ has dimension $\geq 13$.

It follows that a general $P$ is not contained in any multiple of a plane. Indeed, the space of $m.U$ has dimension $3$
whereas for any one, the dimension of the space of length $8$ subschemes $P\subset X\cap m.U$ is $\leq 8$ by \cite{BrianconGrangerSpeder}.

\begin{lemma}
\label{noncol}
In a given irreducible component, a general $P$ does  not contain a colinear subscheme of length $\geq 3$ in a line.
\end{lemma}
\begin{proof}
Start by noting that $P$ is not contained in $U\cup L$ for a plane $U$ and a line $L$. The space of quadrics containing
$U\cup L$ has dimension $2$, whereas $c\leq 7$ so there would be a third quadric containing $P$. One can see that it would
have to contain $L$ so it defines a plane conic $Y\subset U$, meeting $L$, and $P\subset Y\cup L$. But the dimension of the
space of choices of $Y,L$ is $3$ for the plane, $5$ for the conic, $1$ for the intersection point with $L$ and then $2$ for the
direction of $L$ making $11$. Given $Y,L$ the choice of $P$ is discrete
(except in some degenerate cases\footnote{Since $P$ is not contained in
a double plane, $Y$ is not a double line; in the other cases,
singularities of $X\cap (Y\cup L)$ are
always contained in planar singularities of multiplicity $2$
so by \cite{BrianconGrangerSpeder} the dimension of the space of $P$ increases by $1$ at any such point; but existence of the singularity imposes at least one additional condition decreasing the dimension of the space of $Y,L$.}).
The set of such $P$ can therefore not be dense in an irreducible component.

We now show that $P$ cannot have three points colinear in a line $R$, assuming to the contrary that it does.
Choose a point $p\in P$ not contained in $R$
(possible by the paragraph above the lemma).
Let $U$ be the plane spanned by $p$ and $R$. Vanishing on $P\cap R$ and at $p$ impose $4$ conditions on conics of $U$.

In the case $c=6$,
by Lemma \ref{uuprime} with $n\leq 2$ so $4=10-c\geq n+2$,
the residual $P'$ of $P$ with respect to $U$ is 
contained in a line $L$, and we get
$P\subset U\cup L$ contradicting the first paragraph.

In the case $c=7$, by Lemma \ref{uuprime} with $n\leq 2$ so $3=10-c\geq n+1$, we get $P\subset U.U'$.
Both $U$ and $U'$ must contain points not touching $R$.
The residual
$P'$ of $P$ with respect to $U$ has length $\geq 4$, indeed if it were to consist of $3$ points they would have to be colinear
by the $CB(1)$ property but that would give $P\subset U\cup L$.

If $U'$ doesn't contain $R$, the intersection $P\cap (U'\cup R)$ has 
length\footnote{An algebraic argument is needed for the piece of $P$ located at $R\cap U'$; letting $A$ denote its coordinate
algebra, $u$ the equation of $U'$, $f$ the equation of $U$ and $g$ the equation of another plane through $R$, our hypothesis is $fuA=0$
and the local piece of $P\cap (U'\cup R)$ corresponds to $A/guA$. 
Considering the exact sequence 
$$
A/guA \rightarrow A/(fA+gA) \oplus A/uA \rightarrow \cc \rightarrow 0
$$
we see that if the required inequality 
$\ell (A/guA)\geq \ell (A/(fA+gA)) + \ell (fA)$ 
didn't hold we would have $guA=(fA+gA)\cap uA$ and 
$fA\cong A/uA$  hence also $uA \cong A/fA$.
The exact sequence 
$$
0\rightarrow guA \rightarrow uA \rightarrow A/(fA+gA)
$$
becomes $0\rightarrow g(A/fA) \rightarrow A/fA \stackrel{u}{\rightarrow} A/(fA+gA)$ which would give that multiplication by $u$ on
$A/(fA+gA)$ is injective, but that isn't possible since $A$ has finite length. 
}   
at least 
$7$, but since $U'\cup R$ is cut out by
quadrics the $CB(2)$ property of $P$ says that in fact $P\subset (U'\cup R)$ contradicting the first paragraph of the proof.

Suppose $R\subset U'$. Given a residual point lying along $R$,
it cannot correspond to a subscheme leaving $R$ in a direction different from $U'$. For in that case, we could let $U_2$ be the plane contacting this direction, different from $U$ or $U'$, and
applying Lemma \ref{uuprime} again
would give $P\subset U_2U_3$ contradicting the fact that both $U$ and $U'$ contain points of $P$ not on $R$. So, any point
of $P'$ along $R$ corresponds to a point of extra contact with $U'$. We conclude that the residual subscheme of $P\cap U'$ with respect to $R\subset U'$, has length $\geq 2$. Therefore,
$n=1$ conditions suffice to imply vanishing of quadrics on $U'$ so by Lemma \ref{uuprime} this time with $3=10-c\geq n+2$ we
find that the residual of $P$ with respect to $U'$ is contained in a line. This again gives $P$ contained in a plane plus a line,
contradicting the first paragraph of the proof.
\end{proof}

We may now show that the case $c=6$ doesn't contribute
a general point of an irreducible component.
Choose $3$ points of $P$ defining a plane $U$ and apply
Lemma \ref{uuprime} adding
$n\leq 3$ extra
conditions: we get at least one quadric in our family that
has the form $U.U'$. Now if say $U\cap P$ has  length $5$
then the residual would have length $3$ and satisfy $CB(1)$,
therefore it would have to be colinear, contradicting the previous lemma. It follows that $U\cap P$ and $U'\cap P$ both have length $4$. But then, it actually sufficed to add
$n\leq 2$ conditions so we get a line containing the residual, again contradicting Lemma \ref{noncol}. This finishes ruling out the possibility of an irreducible component whose general point imposes $c\leq 6$ conditions on quadrics.

Therefore assume $c=7$. Now $\{ E,P\}$ and $\{ P\}$
have the same dimension which is $\geq 13$. There is a vector space of dimension $10-c=3$
of quadrics passing through $P$.
Let $H_1,H_2,H_3$ denote the elements of a basis.

Here the proof divides into an analysis of two distinct cases; these were
called $(a)$ and $(b)$ in \cite{MS1} refering
to the two cases of Proposition 7.1 from there. Case $(a)$
is when $H_1\cap H_2\cap H_3$ has dimension $0$. It is
a subscheme of length $8$ so we get
$$
P=H_1\cap H_2\cap H_3.
$$
A general such subscheme satisfies $CB(2)$, and
I. Dolgachev pointed out to us that these are called ``Cayley octads''. We shall treat the Cayley octads of case
$(a)$ secondly, since that will use one
part of the discussion of case $(b)$.

\noindent
{\bf Case (b).}

This is when the subscheme $Y=H_1\cap H_2\cap H_3$
contains a pure $1$-dimensional subscheme $Y_1$. Notice that
$Y_1$ is a union of components of the
curve\footnote{Note that $H_i$ cannot all vanish on some plane, otherwise
by $CB(1)$ for the residual $P$ would have to be contained in the plane as we saw previously.}
$H_1\cap H_2$. On the other hand, by Lasker's theorem \cite[p. 314]{EisenbudGreenHarris} if $Y_1$ were equal to $H_1\cap H_2$ then there couldn't be a third quadric vanishing on $Y_1$. Therefore, $Y_1$ is a curve of degree $\leq 3$.

We will now show that $Y_1$ doesn't contain a line.
Suppose to the contrary that $R\subset Y_1$ is a line. Then, all quadrics in our
family contain $R$.

Choose a point $p$ of $P$ not on $R$, let $U$ be the plane through $R$ and $p$, and apply Lemma \ref{uuprime} with $n=2$ to get $P\subset U.U'$.
If $P\cap U'$ has length $\geq 5$, it doesn't have four colinear points so it imposes $5$ conditions on conics, hence we can apply Lemma \ref{uuprime} with $n=1$ and get three residual points in a line, contradicting Lemma \ref{noncol}. 
Therefore $P\cap U$ has length $\geq 4$, however since $P\cap R$ has length $\leq 2$ by Lemma \ref{noncol}, the residual of $P\cap U$ with respect to $R$
has length $\geq 2$. Now, vanishing on $R$ and on $P\cap U$ imposes $5$ conditions on conics of $U$. Thus we may again apply Lemma \ref{uuprime} with $n=1$ and get a residual consisting of $3$ colinear points contradicting Lemma \ref{noncol}. 
This completes the proof that $Y_1$ does not contain a line.

That rules out almost all of the cases listed
in \cite[Lemma 7.4]{MS1}.

A next case is if $Y_1$ is a conic in a plane $U$. Then,
it suffices to impose a single condition, $n=1$ in Lemma \ref{uuprime}, so $3=10-c\geq n+2$ and the residual subscheme consists of at least $3$ points in a line.  This contradicts Lemma \ref{noncol}, so $Y_1$ cannot be a plane conic.

The only remaining possibility for our curve of degree three, is that $Y_1$ could be a rational cubic curve not contained in a plane. It has to be a rational normal cubic, in particular smooth. The restriction of $\Oo _{\pp ^3}(2)$ to the rational curve has degree $6$ so it has seven sections; our three dimensional family of quadrics is therefore the family of all quadrics passing through $Y_1$. They define $Y_1$ schematically, in particular $P\subset Y_1$.

This case will be of interest for our treatment of case $(a)$ below. We have that $P$ is a length $8$ subscheme of the intersection $Y_1\cap X$. For given $Y_1$ the space of choices of $P$ is discrete, and as $Y_1$ moves any $P$ generizes.
The family of such subschemes may therefore be identified with a covering of the space of choices of rational normal cubic $Y_1$. The covering is determined, over a general point, by the choice of $8$ out of the $15$ points in $Y_1\cap X$; or equivalently by the choice of the $7$ complementary points.

The space of choices of $Y_1$ has dimension $12$ (see
\cite[\S 6.2]{MS1}). Therefore, this family cannot constitute an irreducible component of $\{ P\}$. This completes the proof that case $(b)$ cannot happen at a general point of an irreducible component.

\noindent
{\bf Case (a).}

We start this discussion by continuing to look at the above
$12$-dimensional family of subschemes consisting of points in $X\cap Y_1$ for a smooth rational normal cubic curve $Y_1$.

We claim that the family of subschemes, and hence of bundles, obtained in this way is irreducible. This may be seen as follows. Any $6$ points determine the rational normal cubic, so if we move a set of $6$ points around to a different set, we get back to the same rational normal curve and this shows that the monodromy action includes permutations sending any subset of $6$ points to any other one. On the other hand, there is a rational normal curve with first order tangency to $X$, and moving it a little bit induces a permutation of two points keeping the other points fixed. Therefore, the subgroup of the symmetric group contains a transposition. Now since it is $6$-tuply transitive, it contains all the transpositions.
Thus, the monodromy group is the full symmetric group and any group of $8$ points can be moved to any other one. This shows that the family is irreducible.

As was pointed out at the end of Section 6.2 in \cite{MS1}, the space of obstructions at a general point in our family has dimension $1$. The expected dimension is $4c_2-20=12$, so the Zariski tangent space to the moduli space has dimension $13$; however, as noted above any irreducible component has dimension $\geq 13$
because of the existence of the extension. Therefore, a general point of our $12$-dimensional family lies in a smooth open subset of a unique $13$-dimensional irreducible component of
the moduli space $\{ E\}$ (notice here that the spaces $\{ E,P\}$ and $\{ P\}$ are also the same). As our $12$-dimensional family is irreducible by the previous paragraph, this determines a canonical irreducible component of the moduli space.

This discussion corrects an error of notation in the second paragraph of the proof of Lemma 7.6 of \cite{MS1}
where it was stated that the irreducible $12$-dimensional family of Cayley-Bacharach subschemes on the rational normal cubic was inside the type $(a)$ subspace of the moduli space; but that
family  is clearly of type $(b)$. Those phrases should be replaced by the argument of the previous paragraph showing that our $12$-dimensional family is contained in a unique
$13$-dimensional irreducible component of the moduli space, whose general point is of type $(a)$.

We now turn to consideration of the full set of irreducible components, whose general points are of type $(a)$,
that is to say bundles determined by Cayley octad subschemes $P$ (since we showed in the previous part that type $(b)$ cannot lead to a general point of a component).

The argument given in \cite[\S 7.4]{MS1}, using the
incidence variety suggested by A. Hirschowitz, shows that the existence of a canonically defined irreducible component implies irreducibility of the moduli space.

Let us recall her briefly how this works. We look at the full
incidence scheme $\{ X,P\}$ parametrizing smooth quintic hypersurfaces $X$ together with l.c.i.\
subschemes $P\subset X$
of length $8$ satisfying $CB(2)$ of type $(a)$. For a given $P\subset \pp^3$ the space of quintics $X$ containing it is a projective space and these all have the same dimension. So the fibration $\{ X,P\}\rightarrow \{ P\}$ is smooth,
over the base that is an open subset in
the Grassmanian ${\rm Grass}(3,10)$ of $3$-dimensional subspaces of $H^0(\Oo _{\pp^3}(2))$. Thus, the full incidence variety
$\{ X,P\}$ is irreducible. There is a dense open subset of
the space of quintics $\{ X\}$, over which the sets of irreducible components of the fibers don't change locally. Thus, the fundamental group of this open set acts on the set of irreducible components of the fiber $\{ P\} _X$ over a basepoint
$X\in \{ X\}$. This action is transitive, by irreducibility of the full incidence variety. On the other hand, we have described above a canonically defined irreducible component
of $\{ P\}_X$, containing the nearby generalizations of our $12$-dimensional family of subschemes of a rational normal cubic curve. Since it is canonically defined, this component is preserved by the monodromy action. Transitivity now implies that
$\{ P\}_X$ has only a single irreducible component.

This completes the proof of irreducibility for $c_2=8$.
The generic space of obstructions has dimension $1$. That was seen for points on the rational normal cubic curve, at the end of \S 6.2 of \cite{MS1}; however the moduli space has dimension $13$ equal to the expected dimension plus $1$, so the space of obstructions remains $1$-dimensional at a general point.

As the dimension of the moduli space is equal to the expected dimension plus the dimension of the space of obstructions, we
get that the moduli space is generically smooth, and in fact that was already the case at a point of the $12$ dimensional family of subschemes on a rational normal cubic. Since $c=7$ at a general point we  have $h^1(E(1))=8-7-1=0$, to complete the corresponding column of our table.

\subsection{For $c_2=9$}

For the column $c_2=9$, see \cite[Theorem 6.1 and Proposition 7.2]{MS1},  for the dimension $16$
and general obstruction space of dimension $1$.
The proof of \cite[Proposition 7.2]{MS1}
starts out by ruling out, for a general point of an irreducible
component, all cases of \cite[Proposition 7.1]{MS1} except case (d),
for which $c=8$. Thus $h^1(E(1))=9-8-1=0$ for a general
bundle, as we shall also see below.

We give here an alternate argument by dimension count to show that a general bundle in any irreducible component consists of
a collection of $9$ out of the $20$ points on a degree $4$ elliptic curve, intersection of two quadrics, intersected with $X$.

The expected dimension of $\{ E\}$ is $4c_2-20=16$, and
a general $E$ determines a unique\footnote{An easy dimension count rules out the possibility that $P$ be contained in a plane.}
extension hence a unique subscheme $P$ of length $9$.
The dimension of any
irreducible component of $\{ E,P\}$ is $\geq 16$ (notice that it coincides with the value of $3c_2-11$ too).

We first rule out the possibility that $c\leq 7$ for a general point. If there were a three-dimensional family of quadrics passing through $P$ then they cannot intersect transversally in a zero-dimensional subscheme, since that would have length only  $8$ and so be unable to contain $P$. But if the intersection of the three quadrics has a component of positive dimension, then arguing much as in the previous section we can get a contradiction. Indeed, the space of length $9$
subschemes contained in
the intersection of $X$ with two planes has dimension $\leq 3+3+9=15 < 16$, so any time Lemma \ref{uuprime} applies we immediately obtain a contradiction. The remaining case of points on a rational normal curve is ruled out by dimension.

We may therefore assume $c=8$, from which it follows that any irreducible component of $\{ P\}$ has dimension $\geq 16$.
It follows that a general $P$ contains at least $7$ points in general position on $X$. Let us explain the details of this argument, since this kind of dimension count has already been used several times above. Let ${\bf H}\subset \{ P\}$ denote some component of the Hilbert scheme of subschemes we are interested in, that is to say
l.c.i.\ subschemes  $P\subset X$ of length $9$
satisfying $CB(2)$. Let
$$
{\bf I} \subset {\bf H}\times X
$$
be the incidence subscheme, whose fiber over a point
$h\in {\bf H}$ is the subscheme $P_s$ thereby parametrized. Suppose $p_1,\ldots , p_k$ is a collection of distinct points in $X$, and let ${\bf H}(p_1,\ldots , p_k)\subset {\bf H}$ be the closed subscheme parametrizing those $P$ that contain
$p_1,\ldots , p_k$. It may be inductively defined as follows: we have the incidence subvariety
${\bf I}(p_1,\ldots , p_k)\subset
{\bf H}(p_1,\ldots ,  p_k)\times X$,
and for a point $p_{k+1}$ distinct from the other ones,
$$
{\bf H}(p_1,\ldots , p_k,p_{k+1}):= {\rm pr_2}^{-1}(p_k)\subset
{\bf I}(p_1,\ldots , p_k).
$$
By induction we show that for general points $p_i$,
${\bf H}(p_1,\ldots , p_k)$
is nonempty of dimension $\geq 16-2k$
whenever $k\leq 7$.
Assume it is known for $k-1$ but not true for $k$.
That means that the map
${\bf I}(p_1,\ldots , p_{k-1})\rightarrow X$
maps onto a closed subvariety, in other words there is a curve $C\subset X$ depending on $p_1,\ldots , p_{k-1}$ and
containing all of the subschemes parametrized by points of ${\bf H}(p_1,\ldots , p_{k-1})$.
But then the space of such subschemes has dimension $\leq 9-(k-1)$ (by \cite{BrianconGrangerSpeder}), contradicting our inductive hypothesis since $9-(k-1) < 16-2(k-1)$ as $(k-1)<16-9=7$.

After the $7$ points in general position there remain two points. We may conclude that the dimension of a family of subschemes $P$, once the set theoretical locations of the points are known, is $\leq 2$.

We now claim that if $P$ is general, then for a general element $H$ of our family of quadrics passing through $P$, the intersection $H\cap X$ is smooth. The proof is
by a dimension count
of the complementary family.
If the $H\cap X$ is always singular, then
the singular point is a basepoint
(of the linear system on $X$), of which there are finitely many, so it is fixed. Thus, the $H$ are all tangent to $X$ at some point. The space of $2$-dimensional linear systems tangent to $x\in X$ is a Grassmanian ${\rm Grass}(2,\cc^{7})$ of dimension $10$. As the point moves in $X$ we have a $12$ dimensional space of choices of the linear system; and each one of these fixes the set-theoretical location of the points of $P$ so by the previous paragraph, the corresponding
space of $P$ has dimension $\leq 2$, so altogether we obtain
that the family not satisfying our claimed condition has
dimension $\leq 14$. Since any component has dimension $\geq 16$ it follows that the complementary family cannot constitute a component, which proves the claim.

Suppose $V:= H^0(J_{P/\pp^3}(2))\subset \cc ^{10} =H^0(\Oo _{\pp ^3}(2))$ is the two-dimensional space of quadrics passing through our general point $P$.
Then any deformation of the subspace $V\subset \cc^{10}$ lifts to a deformation of $P$. This is because, by the previous claim, we can choose a general element of $V$ corresponding to a quadric $H_1$ such that $H_1\cap X$ is smooth. As the smooth curve deforms, our subscheme $P$ of $(H_1\cap X)\cap H_2$ generalizes since it is uniquely determined just by its multiplicities at each point.

From the above discussion it follows that a general point $P$ in any irreducible component, is obtained by choosing $9$ out of the
$20$ points of $(H_1\cap H_2)\cap X$ for a general pair of quadrics $H_1,H_2$. But since any $8$ points determine the subspace $\langle H_1,H_2\rangle$, the monodromy action on the set of $20$ intersection points is $8$-tuply transtive. By going around a curve $H_1\cap H_2$ with a single simple tangent point to $X$, we get a transposition in the monodromy group; hence it contains all transpositions and it is the full symmetric group. Therefore, the set of choices of $9$ points forms a single orbit under the monodromy group. This completes the proof that there is only one irreducible component of dimension $16$.

The space of obstructions at a general point has dimension $1$,
see the discussion above Theorem 6.1 in \cite{MS1}.
This completes our review of the proof of Proposition \ref{leq9}.

\subsection{For $c_2\geq 10$}
\label{vc2}

We will not be further reviewing the partial result of the case $c_2=10$ that was treated in \cite{MS2}, giving irreducibility of the open subset of the moduli space corresponding to seminatural cohomology as was stated in Proposition \ref{eq10} above, since the argument is more involved and it is the subject of a distinct paper.

On the other hand, it will be useful to discuss in more detail the structure of $V(c_2)$.

\begin{lemma}
\label{vgeq11}
For $c_2\geq 11$, $V(c_2)$ is irreducible of dimension $3c_2-11$
and its general point corresponds to a set of points $P$ in general position with respect to quadrics.
The closure of $V(c_2)$ meets the boundary.
\end{lemma}
\begin{proof}
See \cite[Corollary 7.1]{MS1}, showing that for $c_2\geq 11$,
$\Sigma _{c_2}$ contains an open dense subset $\Sigma ^{10}_{c_2}$ consisting of collections $P$ such that any colength $1$ subscheme imposes vanishing of all quadrics. This is an open subset of the Hilbert scheme of all subschemes $P$ of length $c_2$ so it is smooth, and it further contains an open
dense subscheme where the points of $P$ are distinct. The latter is an open subset of the symmetric product of $X$ so it is irreducible.

The closure of $V(c_2)$ intersects the boundary,
as was discussed in the proof of \cite[Proposition 3.2]{Nijsse}.
Indeed, choose a collection $P_0$ of distinct points that impose vanishing of quadrics but that doesn't satisfy $CB(2)$. Deform this collection in a family $P_t$ such that the general $P_t$ (for $t\neq 0$) satisfies $CB(2)$. Since all elements of the family impose the same number of conditions on quadrics, the space of $Ext$ groups varies in a bundle with respect to the parameter
$t$ and we may choose a family of extensions such that the
general one is locally free. But the special one is not locally free since $P_0$ didn't satisfy $CB(2)$. This family gives a curve in $\Sigma ^{10}_{c_2}$ with parameter $t\neq 0$, whose limiting sheaf at $t=0$ is not locally free: we have a deformation to the boundary.
\end{proof}

\begin{lemma}
\label{v10}
For $c_2=10$, $V(10)$ is irreducible of dimension $3c_2-11=19$ and its general point corresponds
to a subscheme $P$
composed of $10$ general points on a smooth intersection with a quadric $Y=X\cap H$. A general bundle in $V(10)$
has $h^1(E(1))=0$ so any deformation moving away from $V(10)$
will have seminatural cohomology, and only the irreducible component of $M(10)$ constructed in \cite{MS2} contains $V(10)$.
\end{lemma}
\begin{proof}
See \cite[Lemma 3.1]{Nijsse}. General elements of any
irreducible component correspond to subschemes $P$ not contained in a plane, so the irreducible components of $V(10)$ correspond to those of $\Sigma _{10}$ having the same dimension.

By \cite[Corollary 7.1]{MS1}, $\Sigma _{10}$ is pure of dimension $19$. The stratum $\Sigma ^8_{10}$ consisting of
extensions where $P$ lies in the intersection of two quadrics, has dimension $<19$. Indeed, the subscheme $P$ is determined by the two dimensional subspace of
quadrics\footnote{Unless they share a common plane but that case may also be dealt with by a dimension count: $3$ for the choice of plane, plus $4$ for the choice of line, plus at most $7$ for the choice of points in the plane since they would otherwise all be in the plane and then we could ignore the choice of line, plus $1$ for the choice of extension class, comes out to strictly less than $19$.}
and this has dimension $16$, to which we should add $1$ for the space of choices of extension: it comes out strictly less than $19$. Similarly, the dimension of the stratum $\Sigma ^7_{10}$ is strictly less than $19$, and the strata $\Sigma ^c_{10}$ for $c\leq 6$ may be ruled out using our previous line of argument with Lemma \ref{uuprime}.

We conclude that the stratum $\Sigma ^9_{10}$ is dense in
$\Sigma _{10}$. Here the extension class is determined (up to scaling) so $\{ E,P\}$ and $\{ P\}$ are the same, and $\{ P\}$
is an open subset of the space $\{ H,P\}$ parametrizing quadrics $H$ together with $P\subset H\cap X$. The open subset is given
by the conditions that no other quadrics vanish on $P$, and that $P$ satisfies $CB(2)$. But the space $\{ H,P\}$ is irreducible.

Thus, $V(10)$ is irreducible and its general point
parametrizes collections of
$10$ general points on a general smooth quadric section
$Y=X\cap H$. One may now calculate with the standard exact
sequence that for a general
$E\in V(10)$, we have $h^1(E(1))=0$.

Recall by \cite[Corollary 3.5]{MS2}
that the condition of having seminatural cohomology, for bundes in $M(10)$, is equivalent to the conjunction of two
conditions\footnote{We use duality and Euler characteristic
to rewrite the conditions of  \cite[Corollary 3.5]{MS2}.}
$h^1(E(1))=0$ and $h^0(E)=0$. Bundles in $V(10)$ clearly don't satisfy the second condition because $V(10)$ is the locus where $h^0(E)>0$. However, we have seen that a general point of $V(10)$ satisfies the first condition. On the other hand $V(10)$ is pure of dimension $19$ whereas any component of $M(10)$ has dimension $\geq 20$.
Therefore, in any irreducible component of $M(10)$ containing $V(10)$, the general point has $h^0(E)=0$, but also $h^1(E(1))=0$ since it is a generization of the general point of $V(10)$ that satisfies this condition.  Therefore, any irreducible component of
$M(10)$ containing $V$ parametrizes, generically, bundles with seminatural cohomology.

It now follows from the main result of \cite{MS2} (stated as
Proposition \ref{eq10} above) that
any irreducible component of $M(10)$ containing $V(10)$
must be the unique component constructed in \cite{MS2}.
\end{proof}

\section{The double dual stratification}
\label{strata}

Turn now to the proof of the main theorem on the moduli spaces for $c_2\geq 10$.
Our subsequent proofs will make use of O'Grady's techniques \cite{OGradyIrred, OGradyBasic},
as they were recalled and used by Nijsse in \cite{Nijsse}. The main idea is to look at the
boundary of the moduli spaces. His first main observation is the following
\cite[Proposition 3.3]{OGradyBasic}:

\begin{lemma}[O'Grady]
\label{codim1}
The boundary of any irreducible component
(or indeed, of any closed subset) of $M(c_2)$ has pure codimension $1$, if it is nonempty.
\end{lemma}

The boundary is divided up into Uhlenbeck strata corresponding to the ``number of instantons'', which
in the geometric picture corresponds to the number of points where the torsion-free sheaf is not a bundle, counted
with correct multiplicities. A boundary stratum denoted
$M(c_2,c_2-d)$ parametrizes torsion-free sheaves $F$ fitting into an exact sequence of the
form
$$
0\rightarrow F\rightarrow E \stackrel{\sigma}{\rightarrow} S \rightarrow 0
$$
where $E\in M(c_2-d)$ is a stable locally free sheaf of degree $1$ and $c_2(E)=c_2-d$, and $S$ is
a finite coherent sheaf of length $d$ so that $c_2(F)=c_2$.
In this case
$E=F^{\ast\ast}$.
We may think of
$M(c_2,c_2-d)$ as the moduli space of pairs $(E,\sigma )$. Forgetting the quotient $\sigma$
gives a smooth map
$$
M(c_2,c_2-d) \rightarrow M(c_2-d),
$$
sending $F$ to its double dual. The
fiber over $E$ is the Grothedieck ${\rm Quot}$ scheme ${\rm Quot}(E, d)$ parametrizing quotients
$\sigma$ of $E$ of length $d$.

Since we are dealing with sheaves of
degree $1$, all semistable points are stable and our objects have no non-scalar automorphisms.
Hence the moduli spaces are fine, with a universal family existing etale-locally and well-defined up
to a scalar automorphism. We may view the double-dual map as being the relative Grothendieck  ${\rm Quot}$ scheme
of quotients of the universal object $E^{\rm univ}$ on $M(c_2-d)\times X$ over $M(c_2-d)$.
Furthermore,  locally on the ${\rm Quot}$ scheme the quotients are localized near a finite set of points,
and we may trivialize the bundle $E^{\rm univ}$ near these points, so $M(c_2,c_2-d)$
has a covering by, say, analytic open sets which are trivialized as products of open
sets in the base $M(c_2-d)$ with open sets in ${\rm Quot}(E,d)$ for any single choice of $E$.
This is all to say that the map $M(c_2,c_2-d)\rightarrow M(c_2-d)$ may be viewed as a fibration
in a fairly strong sense, with fiber ${\rm Quot}(E,d)$.

Li shows in \cite[Proposition 6.4]{Li} that ${\rm Quot}(E,d)$ is
irreducible with a dense open subset $U$ parametrizing quotients which are given by a collection of
$d$ quotients of length $1$ supported at distinct points of $X$:

\begin{theorem}[Li]
\label{quotli}
Suppose $E$ is a locally free sheaf of rank $2$ on $X$. Then for any $d>0$,
${\rm Quot}(E,d)$ is an irreducible scheme of dimension $3d$, containing
a dense open subset parametrizing quotients $E\rightarrow S$ such that
$S\cong \bigoplus \cc _{y_i}$ where $\cc _{y_i}$ is a skyscraper sheaf of length $1$
supported at $y_i\in X$, and the $y_i$ are distinct. This dense open set
maps to $X^{(d)}-{\rm diag}$ (the space of choices of distinct $d$-uple of points in $X$),
with fiber over $\{ y_i\}$ equal to $\prod _{i=1}^d \pp (E_{y_i})$.
\end{theorem}
\begin{proof}
See Propostion 6.4 in the appendix of \cite{Li}.
Notice right away that
$U$ is an open subset of ${\rm Quot}(F,d)$, and that $U$ fibers over the set $X^{(d)}-{\rm diag}$ of
distinct $d$-uples of points $(y_1,\ldots , y_d)$ (up to permutations). The fiber over
a $d$-uple $(y_1,\ldots , y_d)$ is the product of projective lines $\pp (F_{y_i})$ of quotients of
the vector spaces $F_{y_i}$. As $X^{(d)}-{\rm diag}$ has dimension $2d$, and $\prod _{i=1}^d \pp (F_{y_i})$
has dimension $d$, we get that $U$ is a smooth open variety of dimension $3d$.

This theorem may also be viewed as a consequence of a more precise bound established
by Ellingsrud and Lehn \cite{EllingsrudLehn}, which will be stated as Theorem \ref{el} below,
needed for our arguments in Section \ref{lowest}.
\end{proof}

\begin{corollary}
\label{forward}
We have
$$
{\rm dim} (M(c_2;c'_2) )= {\rm dim} ( M(c_2') )+ 3 (c_2-c_2').
$$
If $M(c_2')$ is irreducible, then $M(c_2;c'_2)$ and hence $\overline{M(c_2;c'_2)}$ are irreducible.
\end{corollary}
\begin{proof}
The fibration $M (c_2;c'_2)\rightarrow M(c'_2)$ has fiber the ${\rm Quot}$ scheme whose dimension
is $3 (c_2-c_2')$ by the previous proposition. Furthermore, these ${\rm Quot}$ schemes are irreducible
so if the base is irreducible, so is the total space.
\end{proof}

Corollary \ref{forward} allows
us to fill in the dimensions of the strata $M (c_2;c'_2)$ in the following table.
The entries in the second column are the expected dimension $4c_2-20$;
in the third column the dimension of $M:= M(c_2)$; and in the
following columns, ${\rm dim} M (c_2,c_2-d)$ for $d=1,2,\ldots $.
The rule is to add $3$ as you go diagonally down and to the right by one.

\bigskip

\begin{table}[h]
\caption{\label{maintable} Dimensions of strata}
\begin{tabular}{|c|c|c|c|c|c|c|c|c|c|c|}
\hline
$c_2$ & e.d.\ & ${\rm dim}(M)$ & ${\scriptstyle d=1}$ & ${\scriptstyle d=2}$ & ${\scriptstyle d=3}$ &
    ${\scriptstyle d=4}$ & ${\scriptstyle d=5}$ & ${\scriptstyle d=6}$ & ${\scriptstyle d=7}$ & ${\scriptstyle d=8}$  \\
\hline
$4$   & $-4$ &  $2$     & $-$   & $-$   & $-$   & $-$   & $-$   & $-$   & $-$   & $-$    \\
$5$   & $0$  & $3$     & $5$   & $-$   & $-$   & $-$   & $-$   & $-$   & $-$   & $-$    \\
$6$   & $4$  & $7$     & $6$   & $8$   & $-$   & $-$   & $-$   & $-$   & $-$   & $-$    \\
$7$   & $8$  & $9$     & $10$  & $9$   & $11$  & $-$   & $-$   & $-$   & $-$   & $-$    \\
$8$   & $12$ & $13$    & $12$  & $13$  & $12$  & $14$  & $-$   & $-$   & $-$   & $-$    \\
$9$   & $16$ & $16$    & $16$  & $15$  & $16$  & $15$  & $17$  & $-$   & $-$   & $-$    \\
$10$  & $20$ & $20$    & $19$  & $19$  & $18$  & $19$  & $18$  & $20$  & $-$   & $-$    \\
$11$  & $24$ & $24$    & $23$  & $22$  & $22$  & $21$  & $22$  & $21$  & $23$  & $-$    \\
$12$  & $28$ & $28$    & $27$  & $26$  & $25$  & $25$  & $24$  & $25$  & $24$  & $26$    \\
\hline
${\scriptstyle \geq 13}$  & ${\scriptstyle 4c_2-20}$ & ${\scriptstyle 4c_2-20}$   & ${\scriptstyle 4c_2-21}$ &
\multicolumn{7}{l|}{${\scriptstyle \leq 4c_2-22}$}  \\
\hline
\end{tabular}
\end{table}

\bigskip

The first remark useful for interpreting this information, is that any irreducible
component of $\overline{M} (c_2)$ must have dimension at least equal to the expected
dimension $4c_2-20$. In particular, a stratum with strictly smaller dimension,
must be a part of at least one irreducible component 
containing a bigger stratum.
For $c_2\geq 11$, we have
$$
{\rm dim} (M (c_2,c_2'))< {\rm dim}(\overline{M} (c_2)) = 4c_2-20.
$$
Hence, for $c_2\geq 11$ the closures $\overline{M (c_2,c_2')}$
cannot themselves form irreducible components of
$\overline{M} (c_2)$, in other words the irreducible components of $\overline{M} (c_2)$ are the
same as those of $M (c_2)$.
Notice, on the other hand, that $\overline{M} (10)$ contains two pieces of dimension $20$,
the locally free sheaves in $M(10)$ and the sheaves in $M(10,4)$ whose double duals come from $M(4)$.

Recall from Proposition \ref{leq9} that the moduli spaces $M(c_2)$ are irreducible for $c_2=4,\ldots , 9$.
It follows from Corollary \ref{forward} that the strata $M(c_2,c_2')$ are irreducible, for any
$c_2'\leq 9$. In particular, the piece $\overline{M (10,4)}$ is irreducible, and its general point,
representing a non-locally free sheaf,
is not confused with any point of
$\overline{M(10)}$. Since the other strata of $\overline{M} (10)$ all have dimension $<20$, it follows that
$\overline{M (10,4)}$ is an irreducible component of $\overline{M} (10)$.
One similarly gets from the table that $\overline{M}(c_2)$ has several irreducible components when $5\leq c_2\leq 9$.

\section{Hartshorne's connectedness theorem}
\label{hartshorne}

Hartshorne proves a connectedness theorem for local complete intersections.
Here is the version that we need.

\begin{theorem}[Hartshorne]
Suppose $Z$ is a local complete intersection of dimension $d$. Then, any nonempty intersection of two irreducible components
of $Z$ has pure dimension $d-1$.
\end{theorem}
\begin{proof}
See \cite{HartshorneConnectedness, SawantHC}.
\end{proof}

\begin{corollary}
\label{meetsing}
If the  moduli space $\overline{M} $ is good, and has two different irreducible components $Z_1$  and $Z_2$
meeting at a point $z$, then $Z_1\cap Z_2$ has codimension $1$ at $z$ and the singular
locus ${\rm Sing}(\overline{M} )$ contains $z$ and has codimension $1$ at $z$.
\end{corollary}
\begin{proof}
If $\overline{M} $ is good, then by Lemma \ref{lci} it as a local complete intersection
so Hartshorne's theorem applies: $Z_1\cap Z_2$ has pure codimension $1$. The intersection of two irrreducible components is
necessarily contained in the singular locus.
\end{proof}

We draw the following conclusions.

\begin{corollary}
\label{meetalongdel}
Suppose, for $c_2\geq 10$, that two different irreducible components $Z_1$ and $Z_2$ of $\overline{M}$
meet at a point $z$, then $z$ is on the boundary. 
\end{corollary}
\begin{proof}
If $z$ is not on the boundary, then by the previous corollary 
it is in a component of the singular locus having codimension $1$ in $M$. 
We have seen in \cite[Theorem 7.1]{MS1} that
for $c_2\geq 10$, a piece of ${\rm Sing}(M)$ having codimension $1$ in $M(c_2)$ has to be in $V(c_2)$, cf Proposition \ref{singularlocus}
above. On the other hand
$V(c_2)$ is irreducible, see Lemmas \ref{vgeq11} and \ref{v10}, so
any such component of 
${\rm Sing}(M)$ has to be equal to $V(c_2)$. 

Recall that ${\rm dim}(V(c_2))=3c_2-11$ whereas the dimension of the moduli space is $4c_2-20$, thus for
$c_2\geq 11$ the singular locus has codimension $\geq 2$,
so the present situation could only occur for $c_2=10$. 

But now by Lemma \ref{v10}, $V(10)$ is contained in only one irreducible component of $M$, the one whose general point parametrizes bundles with seminatural cohomology. So, two distinct components cannot meet along $V(10)$. 
\end{proof}

Next, recall one of Nijsse's theorems, connectedness of the moduli space.

\begin{theorem}[Nijsse]
For $c_2\geq 10$, the moduli space $\overline{M} $ is connected.
\end{theorem}
\begin{proof}
See \cite{Nijsse}, Proposition 3.2. We have reviewed the argument in \cite[Theorem 18.8]{MS3}.
\end{proof}

\begin{corollary}
\label{meetsboundary}
Suppose $Z$ is an irreducible component of $\overline{M} (c_2)$ for $c_2\geq 10$. Then $Z$ meets the boundary in a nonempty subset of
codimension $\leq 1$.
\end{corollary}
\begin{proof}
The codimension $1$ property is given by Lemma \ref{codim1}, so we just have to show that $\overline{Z}$
contains a boundary point.

For $c_2\geq 10$, the first boundary stratum $M(c_2,c_2-1)$ has codimension $1$, so
it must meet at least one irreducible component of
$\overline{M(c_2)}$, call it $Z_0$. Of course if
$Z=Z_0$ we are done. 
Suppose $Z\subset M(c_2)$ is another irreducible component with $c_2\geq 10$. By the connectedness of $\overline{M} (10)$,
there exist a sequence of irreducible components $Z_0,\ldots , Z_k=\overline{Z}$ such that $Z_i\cap Z_{i+1}$ is nonempty.
By Lemma \ref{meetalongdel}, $Z_{k-1}\cap Z_k$ is contained in
the boundary. 
\end{proof}

\section{Seminaturality along the $19$-dimensional boundary strata}
\label{seminat19}

To treat the case $c_2=10$, we will apply the main result of our previous paper.

\begin{proposition}
\label{sn}
Suppose $Z$ is an irreducible component of $M(10)$. Suppose that $\overline{Z}$ contains a point
corresponding to a torsion-free sheaf $F$ with $h^1(F(1))=0$. Then $Z$ is the unique irreducible component
containing the open set of bundles with seminatural cohomology, constructed in \cite{MS2}.
\end{proposition}
\begin{proof}
The locus $V(c_2)$ of bundles with $h^0(E)\neq 0$ has dimension $\leq 19$, so a general point of
$Z$ must have $h^0(E)=0$. The hypothesis implies that a general point has $h^1(E(1))=0$.
Thus, there is a nonempty dense open subset $Z'\subset Z$ parametrizing bundles
with $h^0(E)=0$ and  $h^1(E(1))=0$.
By \cite[Corollary 3.5]{MS2},
these bundles have seminatural cohomology. Thus, our open set is $Z'= M(10)^{\rm sn}$, the moduli
space of bundles with seminatural cohomology, shown to be irreducible
in the main Theorem 0.2 of \cite{MS2} recalled as
Proposition \ref{eq10} above.
\end{proof}

Using Proposition \ref{sn}, and since we know by Corollary \ref{meetsboundary}
that any irreducible component $Z$ meets the boundary in a codimension $1$ subset,
in order to prove irreducibility of $M(10)$,
it suffices to show that the torsion-free sheaves $F$ parametrized by general points on the various irreducible components
of the boundary of $\overline{M(10)}$ have $h^1(F(1))=0$.

The dimension is ${\rm dim}(Z)=20$, so the boundary components
will have dimension $19$. Looking at the line $c_2=10$ in Table \ref{maintable}, we notice that there
are three $19$-dimensional boundary pieces, and a $20$-dimensional piece which must constitute
a different irreducible component. Consider first the $19$-dimensional pieces,
$$
M(10,9), \;\;\;\; M(10,8) \mbox{  and  } M(10,6).
$$
Recall that
$M(10,10-d)$ consists generically of torsion-free sheaves $F$ fitting into
an exact sequence
\begin{equation}
\label{ffs}
0\rightarrow F \rightarrow F^{\ast\ast} \rightarrow S \rightarrow 0
\end{equation}
where $F^{\ast\ast}$ is a general point in the moduli space of stable bundles with $c_2=10-d$,
and $S$ is a general quotient of length $d$.

\begin{proposition}
\label{firstthree}
For a general point $F$ in either of the three boundary pieces $M(10,9)$,
$M(10,8)$ or $M(10,6)$, we have $h^1(F(1))=0$.
\end{proposition}
\begin{proof}
Notice that $\chi (F^{\ast\ast}(1)) =15-c_2(F^{\ast\ast}) \geq 6$ and by stability $h^2(F^{\ast\ast}(1))= h^0(F^{\ast\ast}(-1))=0$,
so $F^{\ast\ast}(1)$ has at least six linearly independent sections. In particular, for a general quotient $S$ of length
$1$, $2$ or $4$, consisting of the direct sum $S=\bigoplus S_x$ of general rank $1$ quotients $E_x\rightarrow S_x$
at $1$, $2$ or $4$ distinct general points $x$,
the map
$$
H^0(F^{\ast\ast}(1))\rightarrow H^0(S)
$$
will be surjective.

For a general point $F^{\ast\ast}$ in either $M(9)$, $M(8)$ or $M(6)$, we have $h^1(F^{\ast\ast}(1))=0$.
These results from \cite{MS1} were recalled in Proposition \ref{leq9}, Table \ref{leq9table}.
The long exact sequence associated to \eqref{ffs} now gives $h^1(F(1))=0$.
\end{proof}

This treats the $19$-dimensional irreducible components of the boundary.
There remains the piece $\overline{M(10,4)}$ which has dimension $20$. This is a separate irreducible
component. It could meet $\overline{M(10)}$ along a $19$-dimensional divisor, and we would like to
show that $h^1(F(1))=0$ for the sheaves parametrized by this divisor. In particular,
we are no longer in a completely generic situation so some further discussion is needed.
This will be the topic of the next section.

\section{The lowest stratum}
\label{lowest}

The lowest stratum is $M(10,4)$, which is therefore closed.
We would like to understand the points in $\overline{M(10)} \cap M(10,4)$.
These are singular, so our main tool will be to look at where the singular locus of $\overline{M} (10)$
meets $M(10,4)$. Denote this by
$$
M(10,4)^{\rm sing} := {\rm Sing}(\overline{M} (10)) \cap M(10,4).
$$
In what follows, we give a somewhat explicit description of the lowest moduli space
$M(4)$.

\begin{lemma}
\label{m4start}
For $E\in M(4)$ we have $h^1(E)=0$, $h^0(E)=h^2(E)=3$, $h^0(E(1)) =11$, and  $h^1(E(1))=h^2(E(1))=0$.
\end{lemma}
\begin{proof}
Choosing an element $s\in H^0(E)$ gives an exact sequence
\begin{equation}
\label{esE}
0\rightarrow \Oo _X \rightarrow E\rightarrow J_{P/X}(1)\rightarrow 0.
\end{equation}
In \cite{MS1} we have seen that $P\subset X\cap L$ is a subscheme of length $4$ in the intersection of $X$ with a line
$L \subset \pp ^3$.
As $P$ spans $L$, the space of linear forms vanishing on $P$ is the same
as the space of linear forms vanishing on $L$, so $H^0(J_{P/X}(1))\cong \cc ^2$.
In the long exact sequence associated to \eqref{esE},
note that $H^1(\Oo _X)=0$, giving
$$
0\rightarrow H^0(\Oo _X)\rightarrow H^0(E)\rightarrow H^0(J_{P/X}(1))\rightarrow 0
$$
hence $H^0(E)\cong \cc ^3$. By duality, $H^2(E)\cong \cc^3$, and the Euler characteristic of $E$ is $6$,
so $H^1(E)=0$.

For $E(1)$, note that $H^2(E(1))=0$ by stability and duality, and \eqref{esE} gives an exact sequence
$$
0\rightarrow H^1(E(1))\rightarrow H^1(J_{P/X}(2))\rightarrow H^2(\Oo _X(1)) \rightarrow 0.
$$
On the other hand, $H^1(J_{P/X}(2)) \cong \cc$ corresponding to the length $4$ of $P$, minus the dimension $3$ of
the space of sections of $\Oo _P(2)$ coming from global quadrics (since the space of quadrics on $L$ has dimension $3$).
This gives $H^1(E(1))=0$.  The Euler characteristic then gives $h^0(E(1))=11$. This is also seen in the first part of
the exact sequence, where $H^0(\Oo _X(1))=\cc^4$ and $H^0(J_{P/X}(2))\cong \cc^7$.
\end{proof}

If $p\in \pp ^3$, let $G\cong \cc ^3$ be the space of linear generators of the ideal of $p$,
that is to say $G:=H^0(J_{p/\pp ^3}(1))$, and consider the natural exact sequence of sheaves on $\pp ^3$
$$
0\rightarrow \Oo _{\pp ^3}(-1)\rightarrow \Oo _{\pp ^3}\otimes G^{\ast} \rightarrow \Rr _p \rightarrow 0.
$$
Here the cokernel sheaf $\Rr _p$ is a reflexive sheaf of degree $1$, and $c_2(\Rr _p)$
is the class of a line. The restriction $\Rr _p|_X$ therefore has $c_2=5$. If $p\in X$, it is torsion-free but
not locally free, giving a point in $M(5,4)$. It turns out that these sheaves account for all of $M(4)$ and $M(5)$.

\begin{theorem}
\label{descrip45}
Suppose $E\in M(4)$. Then there is a unique point $p\in X$ such that $E$ is generated by global sections outside of $p$,
and $\Rr _p|_X$ is isomorphic to the subsheaf of $E$ generated by global sections. This fits into an exact sequence
$$
0\rightarrow \Rr _p|_X \rightarrow E\rightarrow S\rightarrow 0
$$
where $S$ has length $1$, in particular $E\cong (\Rr _p|_X)^{\ast\ast}$. The correspondence $E\leftrightarrow p$
establishes an isomorphism $M(4)\cong X$.

For $E'\in M(5)$ there exists a unique point $p\in \pp ^3-X$ such that $E'\cong \Rr _p|_X$. This correspondence
establishes an isomorphism $\overline{M(5)}\cong \pp ^3$ such that the boundary component $M(5,4)\cap \overline{M(5)}$
is exactly $X\subset \pp ^3$. Note however that $M(5,4)$ itself is bigger and constitutes another irreducible component of
$\overline{M} (5)$.
\end{theorem}
\begin{proof}
Consider the exact sequence \eqref{esE}.
The space $H^0(J_{P/X}(1))$ consists of linear forms on $X$ (or equivalently, on $\pp ^3$), which vanish
along $P$. However, a linear form which vanishes on $P$ also vanishes on $L$. In particular,
elements of $H^0(J_{P/X}(1))$ generate $J_{X\cap L /X}(1)$, which has colength $1$ in $J_{P/X}(1)$.

Let $R\subset E$ be the subsheaf generated by global sections, and let $S$ be the cokernel in the exact sequence
$$
0\rightarrow R \rightarrow E\rightarrow S\rightarrow 0.
$$
We also have the exact sequence
$$
0\rightarrow J_{X\cap L /X}(1)\rightarrow J_{P/X}(1)\rightarrow S\rightarrow 0
$$
so $S$ has length $1$. It is supported on a point $p$. The sheaf $R$ is generated by three global sections
so we have an exact sequence
$$
0\rightarrow {\rm ker} \rightarrow \Oo _X^3\rightarrow R\rightarrow 0.
$$
The kernel is a saturated subsheaf, hence locally free, and by looking at its degree we have ${\rm ker}=\Oo _X(-1)$.
Thus, $R$ is the cokernel of a map $\Oo _X(-1)\rightarrow \Oo _X^3$ given by three linear forms; these
linear forms are a basis for the space of forms vanishing at the point $p$. We see that $R$ is the restriction
to $X$ of the sheaf $\Rr _p$ described above, hence $E\cong (\Rr _p |_X)^{\ast\ast}$.
The map $E\mapsto p$ gives a map $M(4)\rightarrow X$,
with inverse $p\mapsto (\Rr _p |_X)^{\ast\ast}$.

The second paragraph, about $\overline{M(5)}$, is not actually needed later and we leave it to the reader.
\end{proof}

Even though the moduli space $M(4)$ is smooth, it has much more than the expected dimension,
and the space of co-obstructions is nontrivial. It will be useful to understand the
co-obstructions, because if $F\in M(10,4)$ is a torsion-free sheaf with $F^{\ast\ast}=E$
then co-obstructions for $F$ come from co-obstructions for $E$ which preserve the subsheaf $F\subset E$.

\begin{lemma}
\label{m4obs}
Suppose $E\in M(4)$.
A general co-obstruction $\phi : E\rightarrow E(1)$ has generically
distinct eigenvalues with an irreducible spectral variety in ${\rm Tot}(K_X)$.
\end{lemma}
\begin{proof}
It suffices to write down a map $\phi:E\rightarrow E(1)$ with generically distinct eigenvalues
and irreducible spectral variety. To do this, we construct a map $\phi _R:R\rightarrow R(1)$
using the expression $R=\Rr_p |_X$.
The exact sequence defining $\Rr _p$ extends to the Koszul resolution, a long exact sequence
$$
0\rightarrow \Oo _{\pp ^3}(-1)\rightarrow \Oo _{\pp ^3}^3 \rightarrow
\Oo _{\pp ^3}(1)^3 \rightarrow J_{p/\pp ^3}(2)\rightarrow 0.
$$
Thus $\Rr _p$ may be viewed as the image of the middle map.
Without loss of generality, $p$ is the origin in an affine system of coordinates $(x,y,z)$ for $\aaa ^3\subset \pp^3$,
and the coordinate functions are
the three coefficients of the
maps on the left and right in the Koszul sequence.  The $3\times 3$ matrix in the middle is
$$
K:= \left(
\begin{array}{ccc}
 0      &     z    &    -y     \\
 -z      &    0    &     x     \\
  y     &    -x    &     0
\end{array}
\right) .
$$
Any $3\times 3$ matrix of constants $\Phi$ gives a composed map
$$
\phi _R : \Rr _p \hookrightarrow \Oo _{\pp ^3}(1)^3 \stackrel{\Phi}{\rightarrow} \Oo _{\pp ^3}(1)^3 \rightarrow \Rr _p(1).
$$
Use the first two columns of $K$ to give a map $k:\Oo _{\pp ^3}^2\rightarrow \Rr _p$ which is an isomorphism over
an open set. On the other hand, the projection onto the first two coordinates gives a map
$q:\Rr _p\rightarrow \Oo _{\pp ^3}(1)^2$ which is, again, an isomorphism over an open set. The composition of these
two is the map given by the upper $2\times 2$ square of $K$,
$$
qk= K_{2,2}:=
 \left(
\begin{array}{cc}
 0      &     z        \\
 -z      &    0
\end{array}
\right) .
$$
We can now analyze the map $\phi_R$ by noting that $q\phi _R k = K_{2,3} \Phi K_{3,2}$ where $K_{2,3}$ and $K_{3,2}$ are
respectively the upper $2\times 3$ and left $3\times 2$ blocks of $K$. Over the open set where $q$ and $k$ are
isomorphisms,
$$
q\phi _R q^{-1} = q\phi _R k (qk)^{-1} = K_{2,3} \Phi K_{3,2} K_{2,2}^{-1}.
$$
Now
$$
K_{3,2} K_{2,2}^{-1} =
\left(
\begin{array}{cc}
 0      &     z        \\
 -z      &    0    \\
 y & -x
\end{array}
\right)
\cdot
\left(
\begin{array}{cc}
 0      &     -1/z        \\
 1/z      &    0
\end{array}
\right)
=
\left(
\begin{array}{cc}
 1     &     0        \\
 0      &    1    \\
 -x/z & -y/z
\end{array}
\right) .
$$
Suppose
$$
\Phi =
\left(
\begin{array}{ccc}
 \alpha      &     \beta    &    \gamma     \\
 \delta      &    \epsilon    &     \psi     \\
  \chi     &    \theta    &     \rho
\end{array}
\right)
$$
then
$$
q\phi _R q^{-1} =  K_{2,3} \Phi K_{3,2} K_{2,2}^{-1}
$$
$$
=
\left(
\begin{array}{ccc}
 0      &     z    &    -y     \\
 -z      &    0    &     x
\end{array}
\right)
\cdot
\left(
\begin{array}{ccc}
 \alpha      &     \beta    &    \gamma     \\
 \delta      &    \epsilon    &     \psi     \\
  \chi     &    \theta    &     \rho
\end{array}
\right)
\cdot
\left(
\begin{array}{cc}
 1     &     0        \\
 0      &    1    \\
 -x/z & -y/z
\end{array}
\right)
$$
$$
=
\left(
\begin{array}{ccc}
 0      &     z    &    -y     \\
 -z      &    0    &     x
\end{array}
\right)
\cdot
\left(
\begin{array}{cc}
 \alpha   -\gamma x/z   &     \beta   -\gamma y/z      \\
 \delta  -\psi x/z    &    \epsilon  -\psi y/z       \\
  \chi   -\rho x/z  &    \theta -\rho y/z
\end{array}
\right)
$$
$$
=
\left(
\begin{array}{cc}
 z\delta  -\psi x -  y\chi   +\rho xy/z &      z\epsilon  -\psi y +  y\theta -\rho y^2/z      \\
 -z\alpha   +\gamma x-  x\chi   +\rho x^2/z  &   -z\beta   +\gamma y + x\theta -\rho xy/z
\end{array}
\right) .
$$
Notice that the trace of this matrix is
$$
{\rm Tr}(\phi ) = x(\theta -\psi ) + y (\gamma -\chi ) + z (\delta -\beta ),
$$
which is a section of $H^0(\Oo _{\pp ^3}(1))$ vanishing at $p$. A co-obstruction should
have trace zero, so we should impose three linear conditions
$$
\theta = \psi , \;\;\; \chi = \gamma \;\;\; \delta = \beta
$$
which together just say that $\Phi$ is a symmetric matrix. Our expression simplifies to
$$
q\phi _R q^{-1} =
\left(
\begin{array}{cc}
\beta z  -\psi x -  \gamma y  +\rho xy/z &      \epsilon z  -\rho y^2/z      \\
 -\alpha z    +\rho x^2/z  &   -\beta z + \psi x +\gamma y -\rho xy/z
\end{array}
\right) .
$$
Now, restrict $\Rr _p$ to $X$ to get the sheaf $R$, take its double dual to get $E=R^{\ast\ast}$,
and consider the induced map $\phi : E\rightarrow E(1)$. Over the intersection of our open set with $X$,
this will have the same formula. We can furthermore restrict to the curve $Y\subset X$ given by the
intersection with the plane $y=0$. Note that $X$ is in general position subject to the condition that
it contain the point $p$. Setting $y=0$ the above matrix becomes
$$
(q\phi  q^{-1}) |_{y=0} =
\left(
\begin{array}{cc}
\beta z  -\psi x  &      \epsilon z      \\
 -\alpha z    +\rho x^2/z  &   -\beta z + \psi x
\end{array}
\right) .
$$
Choose for example $\beta = \psi = 0$ and $\alpha = \rho = \epsilon = 1$, giving the matrix whose determinant is
$$
{\rm det} \left(
\begin{array}{cc}
0  &     z      \\
x^2/z -z &   0
\end{array}
\right)  = z^2 - x^2 = (z+x) (z-x) .
$$
The eigenvalues of $\phi |_Y$ are therefore $\pm \sqrt{(z+x) (z-x)}$, generically distinct.
For a general choice of the surface $X$, our curve $Y = X\cap (y=0)$ will intersect the planes $x=z$ and $x=-z$
transversally, so the two eigenvalues of $\phi |_Y$ are permuted when going around points in the ramification locus
different from $p$. This
provides an explicit example of $\phi$ for which the spectral variety is irreducible, completing the proof of the lemma.
We included the detailed calculations because they look to be useful if one wants to write down explicitly
the spectral varieties.
\end{proof}

Turn now to the study of the boundary component $M(10,4)$ consisting of torsion-free sheaves in $\overline{M} (10)$ which
come from bundles in $M(4)$.
A point in $M(10,4)$ consists of a torsion-free sheaf $F$ in an exact sequence of the form \eqref{ffs}
$$
0\rightarrow F \rightarrow E \stackrel{\sigma }{\rightarrow} S\rightarrow 0
$$
where $E= F^{\ast\ast}$ is a point in $M(4)$, and $S$ is a length $6$ quotient.

The basic description of the space of obstructions as dual to the space of $K_X$-twisted endomorphisms still holds for torsion-free
sheaves. Thus, the obstruction space for $F$ is ${\rm Hom}^o(F,F(1))^{\ast}$. A co-obstruction
is a map $\phi : F\rightarrow F(1)=F\otimes K_X$ with ${\rm Tr}(\phi )=0$, which is a kind of {\em Higgs field}.
Since the moduli space is good, a point $F$ is in
${\rm Sing}(\overline{M} (10))$ if and only if the obstruction space is nonzero,
that is to say, if and only if there exists a nonzero trace-free $\phi : F\rightarrow F(1)$.

To give a map $\phi$ is the same thing as to give a map $\varphi :E \rightarrow E (1)$
compatible with the quotient map $E \rightarrow S$, in other words fitting into a commutative square with
$\sigma$, for an
induced map $\varphi _S:S\rightarrow S$. The maps $\varphi$, co-obstructions for $E$,
were studied in Lemma \ref{m4obs} above.

Let $\pp (E)\rightarrow X$ denote the Grothendieck projective space bundle. A point in $\pp (E)$
is a pair $(x,s)$ where $x\in X$ and $s: E_x\rightarrow S_x$ is a rank one quotient of the fiber.
Suppose given  a map $\varphi : E \rightarrow E(1)$. We can consider the
{\em internal spectral variety}
$$
{\rm Sp}_E(\varphi )\subset \pp (E)
$$
defined as the set of points $(x,s)\in \pp (E)$ such that there exists a commutative diagram
$$
\begin{array}{ccc}
E_x & \stackrel{\varphi (x)}{\longrightarrow} & E_x \\
\downarrow && \downarrow \\
S_x & \longrightarrow & S_x .
\end{array}
$$
The term `internal' signifies that it is a subvariety of $\pp (E)$ as opposed to the classical spectral
variety which is a subvariety of the total space of $K_X$. Here, we have only given
${\rm Sp}_E(\varphi )$ a structure of closed subset of $\pp (E)$, hence of reduced subvariety. It would
be interesting to give it an appropriate scheme structure which could be non-reduced in case $\varphi$
is nilpotent, but that will not be needed here.

\begin{corollary}
\label{forusual}
Suppose $E\in M(4)$ and $\varphi : E\rightarrow E(1)$ is a general co-obstruction. Then
the internal spectral variety ${\rm Sp}_E(\varphi )$ has a single irreducible component of
dimension $2$.  A quotient $E\rightarrow S$ consisting of a disjoint sum of rank one quotients
$s_i:E_{x_i}\rightarrow S_i$ with $S=\bigoplus S_i$ and the points $x_i$ disjoint, is compatible with $\varphi$
if and only if the points $(x_i,s_i)\in \pp (E)$ lie on the internal spectral variety ${\rm Sp}_E(\varphi )$.
\end{corollary}
\begin{proof}
Notice that $z\in X$ is a point such that $\varphi (z)=0$, then the whole fiber $\pp (E)_z\cong \pp ^1$ is in
${\rm Sp}_E(\varphi )$. In particular, if such a point exists then the map ${\rm Sp}_E(\varphi )\rightarrow X$
will not be finite.

A first remark is that the zero-set of $\varphi$ is $0$-dimensional. Indeed, if $\varphi$ vanished along a divisor
$D$, then $D\in |\Oo _X(n)|$ for $n\geq 1$ and $\varphi : F\rightarrow F(1-n)$. This is possible only if
$n=1$ and $\varphi : F\rightarrow F$ is a scalar endomorphism (since $F$ is stable). However, the trace of
the co-obstruction vanishes, so the scalar $\varphi$ would have to be zero, which we are assuming is not the case.

At an isolated point $z$ with $\varphi (z)=0$, the fiber of the projection
${\rm Sp}_E(\varphi )\rightarrow X$ contains the whole $\pp (E_z)=\pp ^1$. However, these can contribute at most
irreducible components of dimension $\leq 1$ (although we conjecture that in fact these fibers are contained  in
the closure of the $2$-dimensional component so that ${\rm Sp}_E(\varphi )$ is irreducible).

Away from such fibers, the internal spectral variety is isomorphic to the external one, a two-sheeted covering of $X$,
and by Lemma \ref{m4obs}, for a general $\varphi$ the monodromy of this covering interchanges the sheets so it is irreducible.
Thus, ${\rm Sp}_E(\varphi )$ has a single irreducible component of dimension $2$, and it maps to $X$ by a generically finite ($2$ to $1$) map.

The second statement, that a quotient consisting of a direct sum of rank one quotients, is compatible with
$\varphi$ if and only if the corresponding points lie on ${\rm Sp}_E(\varphi )$, is immediate from the
definition.
\end{proof}

\begin{definition}
\label{defusual}
A triple $(E,\varphi , \sigma )$ where $E\in M(4)$, $\varphi : E\rightarrow E(1)$ is a non-nilpotent map, and $\sigma = \bigoplus s_x$
is a quotient composed of six rank $1$ quotients over distinct points, compatible with $\varphi$ as in the previous Corollary \ref{forusual},
leads to an obstructed point $F=F_{(E,\varphi , \sigma )}\in M(10,4)^{\rm sing}$ obtained by setting $F:= \ker (\sigma )$.
Such a point will be called {\em usual}.
\end{definition}

Ellingsrud and Lehn have given a very nice description of the Grothendieck quotient scheme of a bundle of rank $r$ on
a smooth surface.
It extends the basic idea of Li's theorem which we already stated as Theorem \ref{quotli}
above, and will allow us to count dimensions of strata in $M(10,4)$.

\begin{theorem}[Ellingsrud-Lehn]
\label{el}
The quotient scheme parametrizing quotients
of a locally free sheaf $\Oo _X^r$ of rank $r$ on a smooth surface $X$,  located at a given point
$x\in X$, and of length $\ell$, is irreducible of dimension $r\ell -1$.
\end{theorem}
\begin{proof}
See \cite{EllingsrudLehn}. We have given the local version of the statement here.
\end{proof}

In  our case, $r=2$ so the dimension of the local quotient scheme is $2\ell -1$.

A given quotient $E\rightarrow S$ decomposes as a direct sum of quotients $E\rightarrow S_i$ located
at distinct points $x_i\in X$. Order these by decreasing length, and define the {\em length vector}
of $S$ to be the sequence $(\ell _1,\ldots , \ell _k)$ of lengths $\ell _i=\ell (S_i)$ with $\ell _i \geq \ell _{i+1}$.
This leads to a stratification of the ${\rm Quot}$ scheme into strata labelled by length vectors.
By Ellingsrud-Lehn, the dimension of the space of quotients supported at a single (but not fixed) point $x_i$
and having length $\ell _i$, is $2\ell _i + 1$, giving the following dimension count.

\begin{corollary}
For a fixed bundle $E$ of rank $2$,
the dimension of the stratum associated to length vector $(\ell _1,\ldots , \ell _k)$
in the ${\rm Quot}$-scheme of quotients $E\rightarrow S$ with total length $\ell = \sum _{i=1}^k \ell _i$,
is
$$
\sum (2\ell _i + 1) = 2\ell + k.
$$
\end{corollary}

Recall that the moduli space $M(4)$ has dimension $2$, so the dimension of the stratum of $M(10,4)$
corresponding to a vector $(\ell _1,\ldots , \ell _k)$ is $14+k$.
In particular,
$M(10,4)$ has a single stratum $(1,1,1,1,1,1)$ of dimension $20$, corresponding to quotients which are
direct sums of rank one quotients supported at distinct points, and a single stratum $(2,1,1,1,1)$
of length $19$. This yields the following corollary.

\begin{corollary}
\label{intstrata}
If $Z'\subset M(10,4)$ is any
$19$-dimensional irreducible subvariety, then either $Z'$ is equal to the stratum
$(2,1,1,1,1)$,
or else the general point on $Z'$ consists of a direct sum of six rank $1$ quotients
supported over six distinct points of $X$.
\end{corollary}

\begin{proposition}
\label{usual}
The singular locus $M(10,4)^{\rm sing}$ has only one irreducible component of dimension $19$. This irreducible component
has a nonempty dense open subset consisting of the usual points (Definition \ref{defusual}).
For a usual point, the co-obstruction $\varphi$
is unique up to a scalar, so this open set may be viewed as the moduli space of usual triples $(E,\varphi , \sigma )$,
which is irreducible.
\end{proposition}
\begin{proof}
Suppose $Z'\subset M(10,4)^{\rm sing}$ is an irreducible component. Consider the two cases given by
Corollary \ref{intstrata}.

\noindent
{\bf (i)}---If $Z'$ contains an open set consisting of points which are direct sums of
six rank $1$ quotients supported on distinct points of $X$, then this open set parametrizes usual triples.
Furthermore, a point in this open set corresponds to a choice of $(E,\varphi )$ together with six points
on the internal spectral variety ${\rm Sp}_E(\varphi )$. We count the dimension of this piece as follows.

Let $M'(4)$ denote the moduli space of pairs $(E,\varphi )$ with $E\in M(4)$ and $\varphi$ a nonzero co-obstruction for $E$.
The space of co-obstructions for any $E\in M(4)$, has dimension $6$ and the family of these spaces forms a
vector bundle over $M(4)$ (more precisely, a twisted vector bundle twisted by the obstruction class for
existence of a universal family over $M(4)$). Thus, the moduli space of pairs
has a fibration $M'(4)\rightarrow M(4)$ whose fibers are $\pp ^{5}$. In particular, $M'(4)$ is a smooth
irreducible variety of dimension $7$.

For a general such $(E,\varphi )$ the moduli space of usual triples has dimension $\leq 12$, with a unique $12$ dimensional
piece corresponding to a general choice of $6$ points on the unique $2$-dimensional irreducible component of
${\rm Sp}_E(\varphi )$. This gives the $19$-dimensional component of $M(10,4)^{\rm sing}$ mentionned in the statement
of the proposition.

Suppose $(E, \varphi )$ is not general, that is to say, contained in some subvariety of $M'(4)$ of dimension $\leq 6$.
Then, as $\varphi$ is nonzero, even though we no longer can say that it is irreducible,
in any case the internal spectral variety ${\rm Sp}_E(\varphi )$ has dimension $2$
so the space of choices of $6$ general points on it has dimension $\leq 12$, and this contributes at most subvarieties of
dimension $\leq 18$ in $M(10,4)^{\rm sing}$. This shows that in the first case (i) of Corollary \ref{intstrata},
we obtain the conclusion of the proposition.

\noindent
{\bf (ii)}---Suppose $Z'$ is equal to the stratum of $M(10,4)$ corresponding to length vector $(2,1,1,1,1)$.
In this case, we show that a general point of $Z'$ has no non-zero co-obstructions, contradicting the
hypothesis that $Z'\subset M(10,4)^{\rm sing}$ and showing that this case cannot occur.

Fix $E\in M(4)$.
The space of co-obstructions of $E$ has dimension $6$. Suppose $E\rightarrow S_1$ is a quotient of
length $2$. If it is just the whole fiber of $E$ over $x_1$, then it is automatically compatible with
any co-obstruction. However, these quotients contribute only a $2$-dimensional subspace of the
space of such quotients which has dimension $5$ by Ellingsrud-Lehn. Thus, these points don't contribute
general points. On the other hand, a general quotient of length $2$ corresponds to an infinitesimal tangent
vector in $\pp (E)$, and the condition that this vector be contained in ${\rm Sp}_E(\varphi )$ imposes
two conditions on $\varphi$. Therefore, the space of co-obstructions compatible with $S_1$ has
dimension $\leq 4$. Next, given a nonzero co-obstruction in that subspace, a general quotient $E\rightarrow S_2$ of length $1$
will not be compatible, so imposing compatibility with $S_1$ and $S_2$ leads to a space of co-obstructions of dimension $\leq 3$.
Continuing in this way, we see that imposing the condition of compatibility of $\varphi$ with a general quotient $S=S_1\oplus \cdots \oplus S_5$
in the stratum $(2,1,1,1,1)$ leads to $\varphi =0$. Thus, a general point of this stratum has no non-zero co-obstructions as we have
claimed, and
this case (ii) cannot occur.

Hence, the only case from Corollary \ref{intstrata} which can contribute a $19$-dimensional stratum, contributes the single
irreducible component described in the statement of the proposition. One may note that $\varphi$ is uniquely determined
for a general set of six points on its internal spectral variety, since the first $5$ points are general
in $\pp (E)$ and impose linearly independent conditions.
\end{proof}

\begin{corollary}
\label{meetsingood}
Suppose $M(10,4)\cap \overline{M(10)}$ is nonempty. Then it is the unique $19$-dimensional irreducible
component of usual triples in $M(10,4)^{\rm sing}$ identified by Proposition \ref{usual}.
\end{corollary}
\begin{proof}
By Hartshorne's theorem, the intersection
$M(10,4) \cap \overline{M(10)}$ has pure dimension $19$ if it is nonempty. This could also be seen from O'Grady's lemma
that the boundary of $\overline{M(10)}$ has pure dimension $19$. However, any point in this intersection is singular.
By Proposition \ref{usual}, the singular locus $M(10,4)^{\rm sing}$ has only one irreducible component of dimension $19$,
and it is the closure of the space of usual triples.
\end{proof}

If the intersection $M(10,4)\cap \overline{M(10)}$ is nonempty, the torsion-free sheaves $F$ parametrized
by general points satisfy $h^1(F(1))=0$. We show this by a dimension estimate using Ellingsrud-Lehn.
The more precise information about $M(10,4)^{\rm sing}$ given in Proposition \ref{usual}, while not really
needed for the proof at $c_2=10$, will be useful in treating the case of $c_2=11$ in Section \ref{sec11}.

\begin{proposition}
\label{codim2}
The subspace of $M(10,4)$ consisting of points $F$ such that $h^1(F(1))\geq 1$, has codimension $\geq 2$.
\end{proposition}
\begin{proof}
Use the exact sequence
$$
0\rightarrow F\rightarrow E\rightarrow S\rightarrow 0
$$
where $E\in M(4)$. One has $h^1(E(1))=0$ for all $E\in M(4)$, see Lemma \ref{m4start}.
Therefore, $h^1(F(1))=0$ is equivalent
to saying that the map
\begin{equation}
\label{showsurj}
H^0(E(1))\rightarrow H^0(S(1)) \cong \cc ^6
\end{equation}
is surjective.

Considering the theorem of Ellingsrud-Lehn, there are two strata to be looked at: the case of a direct sum
of six quotients of rank $1$ over distint points, to be treated below; and the case of
a direct sum of four quotients of rank $1$ and one quotient of rank $2$. However, this latter stratum
already has codimension $1$, and it is irreducible. So, for this stratum it suffices to note that a general
quotient $E\rightarrow S$ in it leads to a surjective map \eqref{showsurj}, which may be
seen by a classical general position argument, placing first the quotient of rank $2$.

Consider now the stratum of quotients which are the direct sum of six rank $1$ quotients $s_i$ at distinct points
$x_i\in X$.
Fix the bundle $E$. The space of choices of the six quotients $(x_i,s_i)$ has dimension $18$. We claim that the
space of choices such that \eqref{showsurj} is not surjective, has codimension $\geq 2$.

Note that $h^0(E(1))= 11$. Given six quotients $(x_i,s_i)$, if the map \eqref{showsurj} (with $S=\bigoplus S_i$)
is not surjective, then its kernel  has dimension $\geq 6$, so if we choose five additional points $(y_j,t_j)\in \pp (E)$
with $t_j:E_{y_j}\rightarrow T_j$ for $T_i$ of length $1$,
the total evaluation map
\begin{equation}
\label{toteval}
H^0(E(1))\rightarrow  \bigoplus _{i=1}^6 S_i(1) \oplus \bigoplus _{j=1}^5 T_j(1)
\end{equation}
has a nontrivial kernel. Consider the variety
$$
W:= \{ (u, \ldots (x_i,s_i)\ldots , \ldots (y_j,t_j)\ldots ) \mbox{  s.t.  }0\neq  u\in H^0(E(1)), s_i(u)=0, t_j(u)=0 \}
$$
with the nonzero section $u$ taken up to multiplication by a scalar.

Let $Q'_6(E)$ and $Q'_5(E)$ denote the open subsets of the quotient schemes of length $6$ and length $5$ quotients of $E$ respectively,
open subsets consisting of quotients which are direct sums of rank  one quotients over distinct points.
Let $K\subset Q'_6(E)$ denote the locus of quotients $E\rightarrow S$ such that the kernel sheaf $F$ has $h^1(F(1))\geq 1$.
It is a proper closed subset, since it is easy to see that a general quotient $E\rightarrow S$ leads to a surjection \eqref{showsurj}.
The above argument with \eqref{toteval} shows that
$K\times Q'_5(E)\subset p(W)$ where $p:W\rightarrow Q'_6(E)\times Q'_5(E)$ is the projection forgetting the first variable $u$.
Our goal is to show that $K$ has dimension $\leq 16$.

We claim that $W$ has dimension $\leq 32$ and has a single irreducible component of dimension $32$. To see this, start by noting
that
the choice of $u$ lies in the projective space $\pp ^{10}$ associated to $H^0(E(1))\cong \cc^{11}$.

For a section $u$ which
is special in the sense that its scheme of zeros has positive dimension, the locus of choices of $(x_i,s_i)$ and $(y_j,t_j)$
has dimension $\leq 22$, but might have several irreducible components depending on whether the points are on the
zero-set of $u$ or not. However, the space of sections $u$ which are special in this sense, is equal to the space of
pairs $u'\in H^0(E)$, $u''\in H^0(\Oo _X(1))$ up to scalars for both pieces, and this has dimension $2+3=5$,
which is much smaller than the dimension of the space of all sections $u$. Therefore, these pieces don't contribute anything
of dimension higher than $27$.

For a section $u$ which is not special in the sense of the previous paragraph, the space of choices of a single rank $1$ quotient
$(x,s)$ which vanishes on the section, has a single irreducible component of dimension $2$. It might possibly have some pieces
of dimension $1$ corresponding to quotients located at the zeros of $u$ (although we don't think so). Hence, the
space of choices of point in $W$ lying over the section $u$, has dimension $\leq 22$ and has a single irreducible component of
dimension $22$.

Putting these together over $\pp ^{10}$, the dimension of $W$ is $\leq 32$ and it has a single irreducible component of dimension $32$,
as claimed.
Its image $p(W)$ therefore also has dimension $\leq 32$, and has at most one irreducible component of dimension $32$. Denote
this component, if it exists, by $p(W)'$.

Suppose now that $K$ had an irreducible component $K'$ of dimension $17$. Then $K'\times Q'_5(E) \subset p(W)$,
but ${\rm dim}(Q'_5(E))=15$ so
$p(W)'$ would exist and would be equal to $K'\times Q'_5(E)$. However, $p(W)'$ is symmetric under permutation of the
$11$ different variables $(x,s)$ and $(y,t)$, but that would then imply that $P(W)'$ was the whole of $Q'_6(E)\times Q'_5(E)$
which is not the case. Therefore, $K$ must have codimension $\geq 2$.
This completes the proof of the proposition.
\end{proof}

\begin{corollary}
\label{intsn}
Suppose $M(10,4)\cap \overline{M(10)}$ is nonempty. Then
a general point of this intersection corresponds to a torsion-free sheaf with $h^1(F(1))=0$.
\end{corollary}
\begin{proof}
By Hartshorne's or O'Grady's theorem, if the intersection is nonempty then it has pure dimension $19$.
However, the space of torsion-free sheaves $F\in M(10,4)$ with $h^1(F(1))>0$ has dimension $\leq 18$ by
Proposition \ref{codim2}. Thus, a general point in any irreducible component of $M(10,4)\cap \overline{M(10)}$
must have $h^1(F(1))=0$. In fact there can be at most one irreducible component, by Corollary \ref{meetsingood}.
\end{proof}

\section{Irreducibility for $c_2=10$}
\label{sec10}

\begin{corollary}
\label{forseminat}
Suppose $Z$ is an irreducible component of $M(10)$. Then, for a general point $F$ in any irreducible component of the
intersection of $\overline{Z}$ with the boundary, we have $h^1(F(1))=0$.
\end{corollary}
\begin{proof}
By O'Grady's lemma, the intersection of $\overline{Z}$ with the boundary has pure dimension $19$.
By considering the line $c_2=10$ in the Table \ref{maintable},
this subset must be a union of some of the irreducible subsets $\overline{M(10,9)}$, $\overline{M(10,8)}$,
$\overline{M(10,6)}$, and the unique $19$-dimensional irreducible component of
$M(10,4)^{\rm sing}$ given by Proposition \ref{usual}.
Combining Proposition \ref{firstthree} and Corollary \ref{intsn}, we conclude that
$Z$ contains a point $F$ such that $h^1(F(1))=0$. Thus, $h^1(E(1))=0$
for a general bundle $E$ parametrized by a point of $Z$.
\end{proof}

\begin{corollary}
\label{seminat10}
Suppose $Z$ is an irreducible component of $M(10)$. Then the bundle $E$ parametrized by a general point of $Z$ has
seminatural cohomology, and $Z$ is the closure of the irreducible open set $M(10)^{\rm sn}$.
\end{corollary}
\begin{proof}
The closure of $Z$ meets the boundary in a nonempty subset,
by Corollary \ref{meetsboundary}. By the previous Corollary \ref{forseminat}, there exists a point $F$ in $\overline{Z}$
with $h^1(F(1))=0$, thus the general bundle $E$ in $Z$ also satisfies $h^1(E(1))=0$.
By Proposition \ref{sn}, the irreducible moduli space $M(10)^{\rm sn}$ of bundles with seminatural cohomology
is an open set of $Z$.
\end{proof}

\begin{theorem}
\label{irred10}
The moduli space $M(10)$ of stable bundles of degree $1$ and $c_2=10$, is irreducible.
\end{theorem}
\begin{proof}
By Corollary \ref{seminat10}, any irreducible component of $M(10)$ contains a dense open set parametrizing
bundles with seminatural cohomology. By the main theorem of \cite{MS2}, there
is only one such irreducible component.
\end{proof}

\begin{theorem}
\label{tf10}
The full moduli space of stable torsion-free sheaves $\overline{M} (10)$ of degree $1$ and
$c_2=10$, has two irreducible components, $\overline{M(10)}$ and $M(10,4)$ meeting
along the irreducible component of usual triples in $M(10,4)^{\rm sing}$.
These two components have the expected dimension, $20$, hence the moduli space is good and connected.
\end{theorem}
\begin{proof}
Recall that we know $M(10,4)$ is irreducible by the results of \cite{MS1}.
Also $M(10)$ is irreducible. Any component has dimension $\geq 20$,
and by looking at the dimensions in Table \ref{maintable}, these are the
only two possible irreducible components. Since they have dimension $20$ which is the expected dimension,
it follows that the
moduli space is good.

It remains to be proven that these two components do indeed intersect in a nonempty subset, which
then by Corollary \ref{meetsingood}
has to be the irreducible component of usual triples in $M(10,4)^{\rm sing}$. Notice that
Corollary \ref{meetsingood} did not say that the intersection was necessarily nonempty, since
it started from the hypothesis that there was a meeting point. It is
a consequence of Nijsse's connectedness theorem that the intersection is nonempty,
but this may be seen more concretely as follows.

Consider the stratum $M(10,5)$. Recall from
\cite{MS1} that the moduli space $M(5)$ consists of bundles which fit into
an exact sequence of the form
$$
0\rightarrow \Oo _X \rightarrow E \rightarrow J_{P/X}(1)\rightarrow 0,
$$
such that $P= L \cap X$ for $L \subset \pp ^3$ a line. In what follows,
choose $L$ general so that $P$ consists of $5$ distinct points.

The space of extensions
${\rm Ext}^1(J_{P/X}(1), \Oo  _X)$ is dual to ${\rm Ext}^1(\Oo _X, J_{P/X}(2))=H^1(J_{P/X}(2))$.
We have the exact sequence
$$
H^0(\Oo _X(2))\rightarrow H^0(\Oo _P(2))\rightarrow H^1(J_{P/X}(2))\rightarrow 0.
$$
However, $H^0(\Oo _X(2))=H^0(\Oo _{\pp^3}(2))$ and the map to $H^0(\Oo _P(2))$ factors
through $H^0(\Oo _{L}(2))$, the space of degree two forms on $L \cong \pp ^1$,
which has dimension $3$. Hence, the cokernel $H^1(J_{P/X}(2))$ has dimension $2$.
The extension classes which correspond to bundles, are the linear forms on $H^1(J_{P/X}(2))$
which don't vanish on any of the images of the lines in $H^0(\Oo _P(2))$ corresponding
to the $5$ different points. Since $X$ is general,
the collection of $5$ points $X\cap L$ is not in a special 
position in $\pp^1$, so the images of the lines are distinct
in the two dimensional space
$H^1(J_{P/X}(2))$. So we can find a family
of extension classes whose limiting point is an extension which vanishes on one of the lines
corresponding to a point in $P$. This gives a degeneration towards a torsion-free sheaf
with a single non-locally free point, still sitting in a nontrivial extension of the above form.
We conclude that the limiting bundle is still stable, so we have constructed a degeneration
from a point of $M(5)$, to the single boundary stratum $M(5,4)$.

Notice that the dimension of $M(5,4)$ is bigger than that of $M(5)$, so the set of limiting points
is a strict subvariety of $M(5,4)$. We have $\overline{M} (5)=M(5) \cup M(5,4)$, and
we have shown that the closures of these two strata have nonempty intersection.
This fact is also a consequence of the more explicit description of $\overline{M(5)}$
stated in Theorem \ref{descrip45} above (but where the proof was left to the reader).

Moving up to $c_2=10$, it follows that the closure of the stratum $M(10,5)$ intersects $M(10,4)$.
However, $M(10,4)$ is closed, and the remaining strata of the boundary have dimension $\leq 19$,
so all of the other strata in the boundary, in particular $M(10,5)$,
are contained in the closure of the locus of
bundles $\overline{M(10)}$. Thus, $\overline{M(10,5)}\subset \overline{M(10)}$,
but $M(10,4)\cap \overline{M(10,5)}\neq \emptyset$, proving that the intersection
$M(10,4)\cap \overline{M(10)}$ is nonempty.
\end{proof}

\noindent
{\em Physics discussion:}
From this fact, we see that there are degenerations of stable bundles in $M(10)$, near to boundary
points in $M(10,4)$. Donaldson's Yang-Mills metrics then degenerate towards Uhlenbeck boundary points,
metrics where $6$ instantons appear. However, these degenerations go not to all points in $M(10,4)$ but only to ones which are in the
irreducible subvariety $M(10,4)^{\rm sing}\subset M(10,4)$ consisting of points on the internal spectral
variety of a nonzero Higgs field $\varphi : E\rightarrow E\otimes K_X$. It gives a constraint of a global nature
on the $6$-tuples of instantons which can appear in Yang-Mills metrics on a stable bundle $F\in M(10)$.
It would be interesting to understand the geometry of the Higgs field which shows up, somewhat virtually, in the limit.

\section{Irreducibility for $c_2\geq 11$}
\label{sec11}

Consider next the moduli space $\overline{M} (11)$ of stable torsion-free sheaves of degree one and $c_2=11$.
The moduli space is good, of dimension $24$. From Table \ref{maintable}, the dimensions of
the boundary strata are all $\leq 23$, so the set of irreducible components of
$\overline{M} (11)$ is the same as the set of irreducible components of
$M(11)$. Suppose $Z$ is an irreducible component. By Corollary \ref{meetsboundary},
$Z$ meets the boundary in a nonempty subset of codimension $1$, i.e. dimension $23$.
From Table \ref{maintable}, the only two possibilities are $M(11,10)$ and $M(11,4)$.
Note that $M(11,4)$ is closed since it is the lowest stratum; it is irreducible by
Li's theorem and irreducibility of $M(4)$. The stratum $M(11,10)$ is irreducible
because of Theorem \ref{irred10}.

\begin{lemma}
\label{intnonempty}
The intersection $M(11,4)\cap \overline{M(11,10)}$ is a
nonempty subset containing, in particular, points which are torsion-free
sheaves $F'$ entering into an exact sequence of the form
$$
0\rightarrow F' \rightarrow F \rightarrow S_x\rightarrow 0
$$
where $F$ is a usual point of $M(10,4)^{\rm sing}$, $x\in X$ is a general point, and
$F\rightarrow S_x$ is a general rank one quotient.
\end{lemma}
\begin{proof}
Theorem \ref{tf10} shows that the intersection
$M(10,4)\cap \overline{M(10)}$ is nonempty.
It is
the unique $19$-dimensional irreducible component of $M(10,4)^{\rm sing}$, containing the usual points.
Starting with a general point $F\in M(10,4)\cap \overline{M(10)}$
and taking an additional general rank $1$ quotient $S_x$, the subsheaf $F'$ gives
a point in $M(11,4)\cap \overline{M(11,10)}$.
\end{proof}

Let $Y \subset M(10,4)$ be the unique $19$-dimensional
irreducible component of the singular locus $M(10,4)^{\rm sing}$. It contains
a dense open set where the quotient $S$ is a direct sum of six quotients $(x_i,s_i)$ of
rank $1$. Choose a quasi-finite surjection $Y'\rightarrow Y$ such
that $(x_i,s_i)$ are
well defined as functions $Y'\rightarrow \pp (E)$.

Forgetting the quotients and
considering only the bundle $E$ gives a map $Y'\rightarrow M(4)$.
Fix a bundle $E$ in the image of $Y'\rightarrow M(4)$. Let $Y'_E$ denote the fiber of $Y'$ over $E$,
which has dimension $\geq 17$.

We claim that for any $0\leq k\leq 5$, there exists a choice of $k$ out of the $6$ points
such that the map $Y'_E\rightarrow \pp (E)^4$ is surjective. For $k=0$ this is automatic, so
assume that $k\leq 4$ and it is known for $k$; we need to show that it is true for $k+1$ points.
Reorder so that the $k$ points to be chosen, are
the first ones. For a general point $q\in \pp (E)^{k}$, let $Y'_{E,q}$ denote
the fiber of $Y'_E\rightarrow \pp (E)^{k}$ over $q$. We have ${\rm dim}(Y'_{E,q})\geq 17 -3k$.
We get an injection
$$
Y'_{E,q}\rightarrow \pp (E)^{6-k}.
$$
Suppose that the image mapped into a proper subvariety of each factor; then it would map into a
subvariety of dimension $\leq 2(6-k)$, which would give ${\rm dim}(Y'_{E,q})\leq 12-2k$. However,
for $k\leq 4$ we have $12-2k < 17 - 3k$, a contradiction. Therefore, at least one of the projections
must be a surjection $Y'_{E,q}\rightarrow \pp (E)$. Adding this point to our list, gives a list of
$k+1$ points such that the map $Y'_E\rightarrow \pp (E)^{k+1}$ is surjective. This completes the
induction, yielding the following lemma.

\begin{lemma}
\label{general5}
Suppose $Y\subset M(10,4)$ is as above.
Then for a fixed bundle $E\in M(4)$ corresponding to some points in $Y$, and
for a general point in the fiber $Y_E$ over $E$, some $5$ out of the $6$ quotients
correspond to a general point of $\pp (E)^5$.
\end{lemma}
\hfill $\Box$

\begin{lemma}
\label{seventh}
Suppose $F$ is the torsion-free
sheaf parametrized by a general point of $Y$, and let $F'$ be defined by an exact sequence
$$
0\rightarrow F'\rightarrow F\stackrel{(x_7,s_7)}{\longrightarrow} S_7 \rightarrow 0
$$
where $S_7$ has length $1$ and $(x_7,s_7)$ is general (with respect to the choice of $F$) in $\pp (E)$.
Then $F'$ has no nontrivial co-obstructions:
${\rm Hom}(F',F'(1))=0$.
\end{lemma}
\begin{proof}
The space of co-obstructions for the bundle $E$ has dimension $6$. Imposing a condition of compatibility with
a general rank-$1$ quotient $(x_i,s_i)$ cuts down the dimension of the space of co-obstructions by at least $1$.

By Lemma \ref{general5} above, we may assume after reordering
that the first five points $(x_1,s_1),\ldots , (x_5,s_5)$ constitute
a general vector in $\pp (E)^5$. Adding the $7$th general point given by the statement of the proposition, we obtain a general
point $(x_1,s_1),\ldots , (x_5,s_5),(x_7,s_7)$ in $\pp (E)^6$. As this $6$-tuple of points is general with respect to $E$, it imposes
vanishing on the $6$-dimensional space of co-obstructions, giving  ${\rm Hom}(F',F'(1))=0$.
\end{proof}

\begin{corollary}
There exists a point
$$
F'\in \overline{M(11,10)} \cap M(11,4)
$$
in the boundary of $\overline{M} (11)$, such that $F$ is a smooth point of $\overline{M} (11)$.
\end{corollary}
\begin{proof}
By Lemma \ref{seventh}, choosing a general quotient $(x_7,s_7)$ gives a torsion-free sheaf
$F'$ with no co-obstructions, hence corresponding to a smooth point of $\overline{M} (11)$.
By construction we have $F'\in \overline{M(11,10)} \cap M(11,4)$.
\end{proof}

\begin{theorem}
\label{irred11}
The moduli space $\overline{M} (11)$ is irreducible.
\end{theorem}
\begin{proof}
Suppose $Z$ is an irreducible component. Then $Z$ meets the boundary in a codimension $1$ subset;
but by looking at Table \ref{maintable}, there are only two possibilities:
$\overline{M(11,10)}$ and $M(11,4)$. The
co-obstructions vanish for general points of $M(10,4)$ since those correspond to $6$ general
quotients of rank $1$, and the co-obstructions vanish for general points of $M(10)$ by goodness. It follows that there
are no co-obstructions at general points of $\overline{M(11,10)}$ or $M(11,4)$, so each of these
is contained in at most a single irreducible component of $\overline{M} (11)$.
However, in the previous corollary, there is a unique irreducible component
containing $F'$, which shows that the irreducible components containing $\overline{M(11,10)}$ and $M(11,4)$
must be the same. Hence, $\overline{M}(11)$ has only one irreducible component.
\end{proof}

{\em Remark:} Sarbeswar Pal has pointed out to us
a simplified proof for $c_2\geq 11$, avoiding
the use of Lemma \ref{intnonempty}. He observes from the connectedness property and
goodness of the moduli of torsion-free sheaves, that any change
of irreducible component must occur along a codimension $1$ piece of the singular locus. However,
general points of the boundary
components are smooth points of the full moduli space, by an easier version of the previous discussion,
so we can conclude that the singular locus has codimension $\geq 2$. We have nonetheless presented our original proof since it
gives some additional geometrical information on the intersection of the two boundary strata. 

The cases $c_2\geq 12$ are now easy to treat.

\begin{theorem}
\label{irred12}
For any $c_2\geq 12$, the moduli space $\overline{M} (c_2)$ of stable torsion-free sheaves
of degree $1$ and second Chern class $c_2$, is irreducible.
\end{theorem}
\begin{proof}
By Corollary \ref{meetsboundary}, any irreducible component of $\overline{M} (c_2)$
meets the boundary in a subset of codimension $1$. However, for $c_2\geq 12$,
the only stratum of codimension $1$ is $M(c_2,c_2-1)$. By induction on $c_2$,
starting at $c_2=11$, we may assume that $M(c_2,c_2-1)$ is irreducible.
Furthermore, if $E$ is a general point of $M(c_2-1)$ then $E$
admits no co-obstructions, since $M(c_2-1)$ is good. Hence, a general point
$F$ in $M(c_2,c_2-1)$, which is the kernel of a general length-$1$ quotient
$E\rightarrow S$, doesn't admit any co-obstructions either. Therefore,
$\overline{M} (c_2)$ is smooth at a general point of $M(c_2,c_2-1)$. Thus,
there is a unique irreducible component containing $M(c_2,c_2-1)$, which
completes the proof that $\overline{M} (c_2)$ is irreducible.
\end{proof}

We have finished proving our main statement, Theorem \ref{maintheorem} of the introduction:
\newline
{\em for any $c_2\geq 4$, the moduli space $M(c_2)$ of stable vector bundles
of degree $1$ and second Chern class $c_2$ on a very general quintic
hypersurface $X\subset \pp ^3$, is nonempty and irreducible.}

For $4\leq c_2\leq 9$, this is shown in \cite{MS1}.
For $c_2=10$ it is Theorem \ref{irred10}, for
$c_2=11$ it is Theorem \ref{irred11}, and
$c_2\geq 12$ it is Theorem \ref{irred12}. Note that for $c_2\geq 16$
it is Nijsse's theorem \cite{Nijsse}.

It was shown in \cite{MS1} that the moduli space is good for $c_2\geq 10$
(shown by Nijsse for $c_2\geq 13$), and from Table \ref{leq9table} we
see that it isn't good for $4\leq c_2\leq 9$. The  moduli space of
torsion-free sheaves $\overline{M} (c_2)$ is irreducible for $c_2\geq 11$, as may be seen by
looking at the dimensions of boundary strata in Table \ref{maintable}.
Whereas $M(4)=\overline{M}(4)$ is irreducible, the dimensions of the strata in Table \ref{maintable}
imply that $\overline{M}(c_2)$ has several irreducible components for $5\leq c_2\leq 9$, although we haven't answered
the question as to their precise number.
By Theorem \ref{tf10}, $\overline{M} (10)$ has two irreducible components $\overline{M(10)}$ and $M(10,4)$.

\section{An irregularity estimate for \cite{MS1}}
\label{correx}

In this section we provide a correction and improvement to \cite[Lemma 5.1]{MS1} and hence
\cite[Corollary 5.1]{MS1}. There was an error in the proof given in \cite{MS1}.

\begin{lemma}
\label{51bis}
Suppose $X$ is a very general quintic hypersurface in $\pp^3$.
Suppose $s\in H^0(\Oo _X(2))$ is a section which is not the
square of a section of $\Oo _X(1)$. It defines an irreducible spectral covering $Z\subset {\rm Tot}(K_X)$ consisting of
square-roots of $s$. Let $\tilde{Z}$ be a resolution of singularities of $Z$. Then the irregularity of $\tilde{Z}$ is zero,
that is to say $H^0(\tilde{Z},\Omega ^1_{\tilde{Z}})=0$. Hence the dimension of ${\rm Pic}^0({\tilde{Z}})$ is zero.
\end{lemma}
\begin{proof}
The divisor $D$ of zeros of $s$ is reduced since $s$ isn't a square and
in view of the fact that $\Oo _X(1)$ generates ${\rm Pic}(X)$.
Therefore the map $Z\rightarrow X$ is ramified with simple ramification along the smooth points of $D$.
The involution of multiplication by $-1$ acts in the fibers. Choose an equivariant resolution of singularities
${\tilde{Z}}\rightarrow Z$ with an involution $\sigma: {\tilde{Z}}\rightarrow {\tilde{Z}}$ covering the given involution of $Z$.
The irregularity of ${\tilde{Z}}$ is independent of the choice of resolution, so we would like to
show that $H^0(\tilde{Z},\Omega ^1_{\tilde{Z}})=0$.

The map $p:\tilde{Z}\rightarrow X$ induces an exact sequence
$$
0\rightarrow \Oo _X \rightarrow p_{\ast}(\Oo _{\tilde{Z}})\rightarrow Q\rightarrow 0,
$$
with $Q$ a rank $1$ torsion-free sheaf on $X$. The double dual $Q^{\ast\ast}$ is a line bundle $L$.
Using the involution $\sigma$, the above exact sequence splits: $Q$ is the anti-invariant part.
Multiplying together sections of $Q$ gives a map
$$
Q\otimes Q\rightarrow \Oo _X,
$$
which extends by Hartogs to a map
$$
L\otimes L\rightarrow \Oo _X.
$$
Look locally near a smooth point of
$D$ where $X$ has coordinates $(x,y)$ such that $D$ is given by $y=0$, and ${\tilde{Z}}$
has coordinates $(x,z)$ with $y=z^2$.
As a $\cc \{ x,y\}$-module, $Q$ or equivalently $L$ is generated by $z$. The image of the multiplication map is
therefore the submodule generated by $z^2=y$. It is an isomorphism outside of $D$, and to get an isomorphism
it suffices to look off of codimension $2$. This shows that
$$
L\otimes L\stackrel{\cong}{\rightarrow} \Oo _X(-D) \cong \Oo _X(-2)
$$
hence $L\cong \Oo _X(-1)$. It means that $L$ is generated by the linear functions along the fibers of $K_X\rightarrow X$,
restricted back to ${\tilde{Z}}$.

Consider similarly the decomposition into invariant and anti-invariant pieces
$$
p_{\ast} (\Omega ^1_{\tilde{Z}}) = \Ff ^+ \oplus \Ff ^-.
$$
These sheaves are torsion-free, and we have a map $\Omega ^1_X\rightarrow \Ff ^+$.
Again with the local coordinates $x,y$ for $X$ and $x,z$ for
$\tilde{Z}$ near a smooth point of $D$ as above, we have that
$\Omega ^1_{\tilde{Z}}$ is generated by $dx$ and $dz$. As a module over $\cc \{ x,y\}$, $\Ff ^+$ is generated
by $dx$ and $zdz$ or equivalently $dx$ and $dy$. This shows that the map $\Omega ^1_X\rightarrow \Ff ^+$
is an isomorphism on smooth points of $D$. Since $\Ff ^+$ is torsion-free and $\Omega ^1_X$ locally free,
it follows that this map is an isomorphism. We may therefore write
$$
p_{\ast} (\Omega ^1_{\tilde{Z}}) = \Omega ^1_X \oplus \Ff ^-.
$$
Consider now the map $\Omega ^1_X\otimes Q\rightarrow \Ff ^-$. Let $\Gg := (\Ff ^-)^{\ast \ast}$ be the
double dual, and the previous map induces a map
$$
\Omega ^1_X\otimes L\rightarrow \Gg.
$$
Consider again the situation at a smooth point of $D$ using local coordinates. Note that $\Gg$ is generated by
$zdx$ and $dz$, whereas $\Omega ^1_X\otimes L$ is generated by $zdx$ and $zdy=z^2dz = ydz$. Recalling that $L=\Oo _X(-1)$,
we get an exact
sequence
$$
0\rightarrow \Omega ^1_X(-1) \rightarrow \Gg \rightarrow \Bb \rightarrow 0
$$
where $\Bb$ is a sheaf supported on $D$, locally near the smooth points being isomorphic to $\Oo _D$.
This says that $\Gg$ and $\Omega ^1_X(-1)$ are related by an elementary transformation. In particular, we get
$$
0\rightarrow \Gg \rightarrow \Omega ^1_X(-1)(D) = \Omega ^1_X(1) .
$$

The irregularity of $X$ vanishes so $H^0(\Omega ^1_X) =0$. Hence,
$$
H^0(\tilde{Z},\Omega ^1_{\tilde{Z}}) \cong H^0(X, p_{\ast} \Omega ^1_{\tilde{Z}}) \stackrel{\cong}{\rightarrow} H^0(X, \Ff ^-) \hookrightarrow H^0(X, \Gg )
\hookrightarrow H^0(X, \Omega ^1_X(1)).
$$
We have finally shown that there is an injection
$$
H^0({\tilde{Z}}, \Omega ^1_{\tilde{Z}}) \hookrightarrow H^0(X, \Omega ^1_X(1)).
$$
One may show\footnote{For convenience,
here is the argument. The canonical exact sequence
$$
0 \rightarrow \Omega ^1_{\pp^3} \rightarrow \Oo _{\pp^3}(-1)^4 \rightarrow \Oo _{\pp^3} \rightarrow 0
$$
gives rise to
$$
0 \rightarrow H^0(\Omega ^1_{\pp^3}(1)) \rightarrow H^0(\Oo _{\pp^3}^4) \rightarrow H^0(\Oo _{\pp^3}(1))
$$
in which the right map is an isomorphism, so $H^0(\Omega ^1_{\pp^3}(1)) =0$.
We also get $H^1(\pp^3, \Omega ^1_{\pp^3}(-4))=0$, thus the exact sequence
$$
0\rightarrow \Omega ^1_{\pp^3}(-4)\rightarrow \Omega ^1_{\pp^3}(1) \rightarrow \Omega ^1_{\pp^3}(1) |_X
\rightarrow 0
$$
implies $H^0(\Omega ^1_{\pp^3}(1) |_X)=0$.
Now using $H^1(\Oo _X(n))=0$, the exact sequence
$$
0\rightarrow N^{\ast}_{X/\pp^3} (1)= \Oo _X(-4) \rightarrow \Omega ^1_{\pp^3}(1) |_X \rightarrow \Omega ^1_X(1) \rightarrow 0
$$
gives  $H^0(\Omega ^1_X(1))=0$.}
that the right hand space of sections vanishes.
This completes the proof of the lemma.
\end{proof}

Therefore Corollary 5.1 of \cite{MS1} holds, with the improved bound that the dimension is $\leq 9$.
Along the way we have answered \cite[Question 5.1]{MS1}:  in the notation from there, $A=0$.

\section{Example on a degree 6 hypersurface}
\label{example6}

In this section we shall start in the direction of considering hypersurfaces of higher degree,
and consider briefly the case of hypersurfaces of degree $6$. In particular, the notation
differs from that in effect previously.

Here, $X\subset \pp^3$ is a very general hypersurface of degree $6$, which will be denoted $X=X^6$ in the statements
of the main corollaries, for precision. We have $K_X = \Oo_X(2)$. We consider stable rank $2$ vector bundles
$E$ of degree $1$ and more precisely with ${\rm det}(E)=\Oo _X(1)$, and some specified value of $c_2$.

Assume $h^0(E)>0$. Then there is a section, corresponding to a morphism $s:\Oo _X\rightarrow E$. The zeros of $s$ are in codimension $2$,
otherwise it would extend to $\Oo _X(1)\rightarrow E$ contradicting stability. Therefore, $s$ fits into an exact sequence of the
usual form
\begin{equation}
\label{theext}
0\rightarrow \Oo _X\rightarrow E \rightarrow J_{P/X}(1)\rightarrow 0,
\end{equation}
where $P\subset X$ is a local complete subscheme of dimension $0$. By the general theory, $P$ satisfies the condition $CB(L^{-1} \otimes M\otimes K_X)$
where $L=\Oo _X$, and $M=\Oo _X(1)$. In other words, $P$ is a $CB(3)$ subscheme.

Notice that $c_2(\Oo _X\oplus \Oo _X(1))=0$ by the product formula for Chern polynomials; therefore in the above extension, we have $c_2(E)= |P|$.

In our examples, we will consider the case $c_2=11$, and give two different kinds of $11$-point $CB(3)$ subschemes.

Before getting to these, let us note some general things about the deformation theory. Our bundle satisfies $E^{\ast}=E(-1)$, so
$$
{\rm End}(E)=E^{\ast}\otimes E \cong E\otimes E(-1).
$$
The decomposition ${\rm End}(E)={\rm End}^0(E)\oplus \Oo _X$ into the trace-free plus the central part, corresponds to the
decomposition
$$
E\otimes E(-1)= {\rm Sym}^2(E)(-1)\oplus \bigwedge ^2(E)(-1).
$$
Let us denote for short $V:= {\rm Sym}^2(E)(-1)$.
The deformation theory of $E$ as a bundle with fixed determinant is governed by $H^{\ast}(V)$. Notice that if $E$ is stable, it has no endomorphisms
except the scalars,
so $H^0(V)=0$. We may also apply Serre duality noting that $V$ is self-dual and recalling $K_X=\Oo _X(2)$.
The space of infinitesimal deformations is
$$
{\rm Def}(E)=H^1(V)\cong H^1(V(2))^{\ast}
$$
and the space of obstructions is
$$
{\rm Obs}(E)= H^2(V) \cong H^0(V(2))^{\ast}.
$$

Let $2P$ denote the subscheme defined by the square of the ideal of $P$, so $J_{2P/X}= (J_{P/X})^2$.  We have an exact sequence
\begin{equation}
\label{esV}
0\rightarrow E(-1)\rightarrow V \rightarrow J_{2P/X}(1)\rightarrow 0
\end{equation}
and hence
\begin{equation}
\label{esV2}
0\rightarrow E(1)\rightarrow V(2) \rightarrow J_{2P/X}(3)\rightarrow 0.
\end{equation}

\subsection{Points on the rational normal cubic}

The first case is when $C\subset \pp ^3$ is a general rational normal cubic, and $P\subset X\cap C$ is a collection of $11$ points.
This exists since $C\cap X$ consists of $18$ distinct points and we may choose $11$ of them.

Notice that $C\cong \pp^1$ and $\Oo _{\pp^3}(1)|_C=\Oo _C(3p)$ for any
point $p\in C$, that is to say it is a line bundle of degree $3$. Thus, $\Oo _{\pp^3}(3)|_C=\Oo _C(9p)$ has degree $9$.
If $P'\subset P$ is any collection of $10$ points, a section of $\Oo _{\pp^3}(3)$ vanishing on $P'$ must vanish on $C$, hence it must
vanish on $P$. The sections of $\Oo _X(3)$ are all restrictions of sections of $\Oo _{\pp^3}(3)$, so this proves that $P$ satisfies the property $CB(3)$.

The space of extensions of $J_{P/X}(1)$ by $\Oo _X$ is dual to $H^1(J_{P/X}(3))$, which in turn is the cokernel of
\begin{equation}
\label{resP}
H^0(\Oo _X(3))\rightarrow H_0(P,\Oo _P(3))\cong \cc ^{11} .
\end{equation}
As we have seen above, a section of $H^0(\Oo _X(3))$ vanishing on $P$ corresponds to a section of $H^0(\Oo _{\pp^3}(3))$ vanishing on $C$.
One may calculate by hand that the map
$$
\cc ^{20} = H^0(\Oo _{\pp^3}(3))\rightarrow H^0(\Oo _C(9p)) = \cc ^{10}
$$
is surjective. Indeed, the image of $H^0(\Oo _{\pp^3}(1))$ consists of the sections which may be written as $1,t,t^2, t^3$ for an affine coordinate $t$ on $C\cong \pp^1$
with pole at the point $p$. Then,  monomials of degree $3$ in these sections give all of the monomials $1,t,\ldots , t^9$.

From this surjectivity we get that the kernel is $\cc^{10}$. Thus, the kernel of the map \eqref{resP} is $\cc^{10}$ so the image of the map also has dimension
$10$. Finally, we get that the cokernel of \eqref{resP} has dimension $1$. We have shown that $Ext^1(J_{P/X}, \Oo _X)$ has dimension $1$.
Therefore, a given subscheme $P$ gives rise to only one bundle since
scaling of the extension class doesn't change the isomorphism class of the bundle.

For the other direction, we claim that $h^0(E)=1$. Consider the exact sequence
$$
0\rightarrow H^0(\Oo _X)\rightarrow H^0(E)\rightarrow H^0(J_{P/X}(1))\rightarrow H^1(\Oo _X)=0.
$$
Given a section of $H^0(\Oo _X(1))$ vanishing on $P$, it comes from a section of $H^0(\Oo _{\pp^3}(1))$ which, by the same argument as previously, vanishes on $C$.
If the section is nonzero, that would say that $C$ is contained in a plane, which however is not the case. Therefore,
$H^0(J_{P/X}(1))=0$ and $\cc \cong H^0(\Oo _X)\stackrel{\cong}{\rightarrow} H^0(E)$. We get $h^0(E)=1$ as claimed.

In particular, for a given bundle $E$, the choice of section $s$ is unique up to a scalar, so the subscheme $P$ is uniquely determined.

By these arguments, we conclude that the space of bundles $E$ in this case is isomorphic to the space of choices of subscheme $P\subset C\cap X$.

Now, given $P\subset C\cap X$ of length $11$, we claim that $C$ is the only rational normal curve passing through $P$. Indeed, suppose $C'$ were another one.
Note that $C'$ is cut out by conics. If $Q\subset \pp^3$ is a conic containing $C'$ then $Q\cap C$ is either equal to $C$, or has length $6$; the latter
case can't happen so $C\subset Q$. Thus, any conic containing $C'$ also contains $C$, which shows that $C=C'$.

The dimension of the space of subschemes $P$ in this case is therefore equal to the dimension of the space
$PGL(4)/PGL(2)$
of rational normal cubic curves,  which is $15-3=12$. This completes the proof of the following proposition:

\begin{proposition}
The space of bundles $E$ fitting into an exact sequence of the form \eqref{theext}, where $P$ is a length $11$ subscheme of $C\cap X$ for $C$ a rational
normal cubic in $\pp^3$, has dimension $12$.
\end{proposition}

\begin{lemma}
Suppose $E$ is a bundle fitting into an exact sequence of the form \eqref{theext}, where $P$ is a length $11$ subscheme of $C\cap X$ for $C$ a
general rational
normal cubic in $\pp^3$. Then $h^1({\rm End}^0(E))=h^1(V)=12$.
\end{lemma}
\begin{proof}
Use the exact sequence \eqref{esV}. The first step is to calculate
$h^1(E(-1))$. Note that \eqref{theext} gives the following sequence,
using that $h^1(\Oo _X(n))=0$ for any $n$
as well as $H^2(J_{P/X}(n))=H^2(\Oo _X(n))$:
$$
0\rightarrow H^1(E(-1))\rightarrow H^1(J_{P/X})\rightarrow H^2(\Oo _X(-1))\rightarrow
H^2(E(-1))\rightarrow H^2(\Oo _X)\rightarrow 0.
$$
Now $H^2(E(-1))$ is dual to $H^0(E(2))$ which itself fits into the sequence
$$
0\rightarrow H^0(\Oo _X(2))\rightarrow H^0(E(2))\rightarrow H^0(J_{P/X}(3))\rightarrow 0.
$$
We have $H^0(J_{P/X}(3))\cong H^0(J_{C/{\pp ^3}}(3))={\rm ker}(H^0(\Oo _{\pp^3}(3))\rightarrow H^0(\Oo _C(9p)))$. The latter map is
surjective from $\cc ^{20}$ to $\cc ^{10}$ so its kernel has dimension $10$. This gives
$h^0(J_{P/X}(3))=10$. Also $h^0(\Oo _X(2))= 10$ so $h^2(E(-1))=h^0(E(2))=20$. We have $h^2(\Oo _X)= h^0(\Oo _X(2))=10$ and
$h^2(\Oo _X(-1))= h^0(\Oo _X(3))=20$. Finally, $H^1(J_{P/X})$ is just $\cc ^{11}$ modulo $H^0(\Oo _X)=\cc$ so $h^1(J_{P/X})=10$.
The alternating sum from the above sequence vanishes, saying now that
$$
h^1(E(-1)) - 10 + 20 - 20 + 10 = 0,
$$
so $h^1(E(-1))=0$.

The long exact sequence associated to \eqref{esV} starting with $H^1(E(-1))=0$ now gives
$$
0\rightarrow H^1(V)\rightarrow H^1(J_{2P/X}(1))\rightarrow H^2(E(-1))\rightarrow H^2(V)\rightarrow H^2(\Oo _X(1))\rightarrow 0.
$$
As we have seen above, $h^2(E(-1))=20$. It is also easy to see that $h^0(J_{2P/X}(1))=0$ (we will in fact see this for $J_{2P/X}(3)$ below),
so noting that the length of $2P$ is $33$ we get $h^1(J_{2P/X}(1))= 33 -h^0(\Oo _X(1))=29$. Putting these together and using
$h^2(\Oo _X(1))=h^0(\Oo _X(1))=4$ we get
$$
h^1(V) - 29 + 20 - h^2(V) +4 = 0,
$$
so
$h^1(V)-h^2(V) = 5$. This is the expected dimension of the moduli space.

Next, by duality $h^2(V)=h^0(V(2))$ which we can calculate using the sequence \eqref{esV2}. We have
$$
0\rightarrow H^0(E(1))\rightarrow H^0(V(2))\rightarrow H^0(J_{2P/X}(3)) .
$$
We claim that $H^0(J_{2P/X}(3))=0$. To see this, consider a quadric surface $Q\subset \pp^3$ containing $C$. We have $Q\cong \pp^1\times \pp^1$ and
$C$ is a divisor of bidegree $(1,2)$ on $Q$. On the other hand, $\Oo _Q(1)$ has bidegree $(1,1)$. Suppose we have a section $u$ of $H^0(\Oo _X(3))=H^0(\Oo _{\pp^3}(3))$
vanishing on the  $2P$ (recall that $2P$ is the subscheme of $X$ defined by the square of the ideal of $P$).
We have seen already above that vanishing on $P$ implies that it vanishes on $C$. Therefore $u|_Q$ is a section of the
bundle of bidegree $(3,3)-(1,2)=(2,1)$. The intertsection of $2P$ with $Q$ consists of a collection of double points transverse to $C$ at the points of $P$,
so it imposes again a single condition on the section $u$ considered as a section of $\Oo _Q(2,1)$. The restriction of $\Oo _Q(2,1)$ to $C$
is a line bundle on $C\cong \pp^1$ of degree equal to the intersection number $(2,1).(1,2)=5$. Therefore, a section there which vanishes on $11$ points
has to vanish. This says that our section of bidegree $(2,1)$ again vanishes on $C$, so it is a section of a bundle of bidegree $(1,-1)$; but that is not effective
so this section has to vanish. This proves that our section $u|_Q$ vanishes. Therefore, $u$ may be viewed as a section of $\Oo _{\pp^3}(3)(-Q)=\Oo _{\pp^3}(1)$.
The remaining pieces of the
double points composing $2P$ give conditions of vanishing again at all the points of $P$ for this section of $\Oo _{\pp^3}(1)$,
but as $C$ is not contained in a plane, it implies that the section vanishes. This completes the proof that $H^0(J_{2P/X}(3))=0$.
We conclude from the previous exact sequence that
$$
h^2(V)=h^0(V(2))= h^0(E(1)).
$$
Now use the sequence
$$
0\rightarrow H^0(\Oo _X(1))\rightarrow H^0(E(1))\rightarrow H^0(J_{P/X}(2))\rightarrow 0.
$$
As usual, $H^0(J_{P/X}(2))$ is isomorphic to the kernel of the restriction map
$$
\cc ^{10}=H^0(\Oo _X(2))\rightarrow H^0(\Oo _C(6p))=\cc ^7
$$
and this restriction map is surjective, so its kernel has dimension $3$. We get
$$
h^0(E(1))=4+3=7.
$$
Thus, $h^2(V)=7$, and putting this together with the formula that the expected dimension is $5$, we have finally shown $h^1(V)=12$. This proves the lemma.
\end{proof}

Even though there is a $7$-dimensional obstruction space, we have constructed a $12$-dimensional family; it follows that all of the obstructions vanish
and a general point lies in a generically smooth irreducible component of dimension $12$.

\begin{corollary}
\label{component12}
The space of bundles $E$ fitting into an exact sequence of the form \eqref{theext}, where $P$ is a length $11$ subscheme of $C\cap X$ for $C$ a rational
normal cubic in $\pp^3$, consists of a single $12$-dimensional generically smooth irreducible component of the moduli space $M_{X^6}(2,1,11)$ of stable bundles of rank $2$, degree $1$ and $c_2=11$ on
our degree $6$ hypersurface $X=X^6$.
\end{corollary}
\begin{proof}
In order to understand how many irreducible components are produced by this construction,
we should investigate the monodromy of the set of choices of $11$ out of the $18$ points of $C\cap X$, as $C$ moves.
A choice of $6$ points determines the rational normal cubic $C$,
so any $6$ points can be moved to any $6$ other ones. Therefore, the monodromy action is $6$-tuply transitive. On the other hand, it contains a transposition, since we can move $C$ around a choice of curve that is simply tangent to $X$ at one point. Therefore, the monodromy group contains all transpositions, hence it is the full symmetric group on $18$ elements.
It acts transitively on the set of choices of $11$ out of the $18$ intersection points, so our construction produces a single irreducible component.
\end{proof}

\subsection{Points on a plane}

The other construction we have found for $CB(3)$ subschemes is to take $11$ points in a plane. Let $H$ be a plane in general position with respect to
$X$, and let $Y=X\cap H$. Let $P$ consist of a general collection of $11$ points in $Y$.

Suppose $P'\subset P$ is a subset of $10$ points. The map $H^0(\Oo _H(3))\rightarrow H^0(\Oo _Y(3))$ is injective (since $Y$ is a curve of degree $6$
in the plane $H$), so a general collection of $10$ points imposes independent conditions on $H^0(\Oo _H(3))$. As $h^0(\Oo _H(3))=10$, it means that
$H^0(J_{P'/H}(3))=0$, hence a section of $H^0(\Oo _{\pp^3}(3))$ vanishing on $P'$, has to vanish on $H$. In particular it vanishes on $P$, proving the
$CB(3)$ property for $P$. This also gives the formula
$$
H^0(J_{P/X}(3))\cong H^0(\Oo _X(2)) = \cc ^{10}.
$$

Consider next the space of choices of extension \eqref{theext}. As
$$
{\rm dim}({\rm Ext}^1(J_{P/X}(1), \Oo _X))=
h^1(J_{P/X}(3)) = 11 - 20 + h^0(J_{P/X}(3)) = 1,
$$
whereas scalar multiples of an extension class give the same bundle,
it means that for a given $P$ there is a single corresponding bundle.
On the other hand, we have $h^0(J_{P/X}(1))=1$ since $P$ is contained in a plane,
so $h^0(E)=2$. This means that for a given bundle $E$, the space of choices of section $s$ (modulo scaling)
leading to the subscheme $P$, has dimension $1$. Hence the dimension of the space of bundles obtained by this construction is
one less than the dimension of the space of subschemes:
$$
{\rm dim} \{ E\} = {\rm dim} \{ P\} -1.
$$
Count now the dimension of the space of choices of $P$: there is a three dimensional space of choices of the plane $H$, and for each one we have an $11$ dimensional
space of choices of the subscheme $P$ of $11$ points in $Y$. This gives ${\rm dim} \{ P\}=3+11=14$, so
${\rm dim} \{ E\}  = 13$. Altogether, we have constructed a $13$ dimensional family of stable bundles. It follows that this family must be in at least one irreducible
component distinct from the $12$-dimensional component constructed above. This proves the following theorem:

\begin{theorem}
\label{twocomponents}
For a very general degree $6$ hypersurface $X^6\subset \pp^3$,
the moduli space $M_{X^6}(2,1,11)$ contains a generically smooth $12$ dimensional component from Corollary \ref{component12},
and contains at least one irreducible component
of dimension $\geq 13$. In particular, it is not irreducible.
\end{theorem}

The general bundle in our $13$-dimensional family may be viewed as an elementary transformation \cite{MaruyamaET, MaruyamaTransform}.
A general line bundle $L$
of degree $11$ on $Y$ has a $2$-dimensional space of sections and the two sections generate $L$.
If $j:Y\hookrightarrow X$ denotes the inclusion then we get a bundle
$E$, elementary transformation of $\Oo _X^2$, fitting into exact sequences
$$
0\rightarrow E(-1)\rightarrow \Oo _X^2 \rightarrow j_{\ast}(L)\rightarrow 0,
$$
$$
0\rightarrow  \Oo _X^2 \rightarrow E \rightarrow j_{\ast}(L^{\ast})(1)\rightarrow 0.
$$
This shows that $E$ determines $Y$ and $L$.
Since $Y$ has genus $10$, the space of choices of hyperplane plus choice of $L$ has dimension $3+10=13$.
One may see that these bundles are the same as the previous ones, indeed the zeros of a section of our elementary transformation
$E$ are the same as those of the corresponding section of $L$. This gives an alternate
canonical viewpoint on our second construction of bundles that should be useful for understanding the obstruction map.

We conjecture that the rational normal case and the planar case cover all of $M_{X^6}(2,1,11)$. More precisely:

\begin{conjecture}
The $13$-dimensional family constructed in the present subsection
constitutes a full irreducible component of $M_{X^6}(2,1,11)$; this component is non-reduced and obstructed.
Together with the $12$-dimensional generically smooth 
component constructed in the previous subsection, 
these are the only irreducible components of
$M_{X^6}(2,1,11)$. In particular, $h^0(E)>0$ for any stable bundle with $c_2=11$.
\end{conjecture}

There doesn't seem to be an easy direct proof of the property $h^0(E)>0$. The Euler-characteristic consideration does give $h^0(E(1))>0$
so any $E$ has to be in an extension of $\Oo (-1)$ by $J_{P/X}(2)$ with $P$ satisfying $CB(5)$.
If this conjecture is true, it would imply that any $CB(5)$ subscheme of length $21$ contained in $X^6$, would have to be contained in a quadric hypersurface.
We didn't find a proof of that, but we couldn't find any constructions that weren't contained in quadric hypersurfaces either,
leading to the conjecture.

\bibliographystyle{amsplain}

\begin{thebibliography}{A}

\bibitem{Arrondo}
E. Arrondo. A home-made Hartshorne-Serre correspondence.
Arxiv preprint math/0610015 (2006).

\bibitem{Barth}
W. Barth. Moduli of vector bundles on the projective plane.
Inventiones Math. 42 (1977), 63-91.

\bibitem{BrianconGrangerSpeder}
J. Brian\c{c}on, M. Granger, J.-P. Speder.
Sur le sch\'ema de Hilbert d'une courbe plane.
{\em Ann. Sci. E.N.S.} {\bf 14} (1981), 1-25.

\bibitem{EisenbudGreenHarris}
D. Eisenbud, M. Green, J. Harris. Cayley-Bacharach theorems and conjectures. {\em Bull. Amer. Math. Soc.} {\bf 33} (1996),
295-324.

\bibitem{EllingsrudLehn}
G. Ellingsrud, M. Lehn. Irreducibility of the punctual quotient scheme
of a surface. Arkiv f\"or Matematik {\bf 37} (1999), 245-254.

\bibitem{Gieseker}
D. Gieseker. On the moduli of vector bundles on an algebraic surface. Ann. of Math.,
Volume 106 (1977), 45-60.

\bibitem{GiesekerCons}
D. Gieseker. A construction of stable bundles on an algebraic surface. J. Diff. Geom.,
Volume 27 (1988), 137-154.

\bibitem{GiesekerLi}
D. Gieseker, J. Li. Irreducibility of moduli of rank $2$ bundles on algebraic surfaces.
J. Diff. Geom., Volume 40 (1994), 23-104.

\bibitem{GomezThesis}
T. G\'omez. Irreducibility of the moduli space of vector bundles
on surfaces and Brill-Noether theory on singular
curves. Ph.D. thesis, Princeton University, \verb+arxiv:alg-geom/9710029+ (1997).

\bibitem{GriffithsHarris}
P. Griffiths, J. Harris. Residues and zero-cycles on algebraic varieties.
Ann. of Math. 108 (1978), 461-505.

\bibitem{HartshorneConnectedness}
R. Hartshorne. Complete intersections and connectedness. Amer. J. Math. {\bf 84} (1962), 497-508.


\bibitem{HuybrechtsLehn}
D. Huybrechts, M. Lehn. The geometry of moduli spaces of sheaves.
Aspects of Mathematics 31, Max Planck Institute (1997).


\bibitem{Langer}
A. Langer. Lectures on torsion-free sheaves and their moduli. {\em Algebraic cycles,
sheaves, shtukas, and moduli.} {\sc Trends in Math.}, Birkh\"auser (2008), 69-103.

\bibitem{Li}
J. Li. Algebraic geometric interpretation of Donaldson's polynomial invariants. {\em J. Diff.
Geom.} {\bf 37} (1993), 417-466.

\bibitem{Maruyama}
M. Maruyama. Stable vector bundles on an algebraic surface. Nagoya Math. J., Volume 58 (1975), 25-68.

\bibitem{MaruyamaET}
M. Maruyama. On a family of algebraic vector bundles.
Number Theory, Algebraic Geometry, Commutative Algebra, in Honor of Yasuo Akizuki (1973),
95-146.

\bibitem{MaruyamaTransform}
M. Maruyama. Elementary transformations in the theory of algebraic vector bundles.
Algebraic Geometry (La Rabida), Springer L.N.M. 961 (1982), 241-266.

\bibitem{Mestrano}
N. Mestrano. Sur le espaces de modules de fibr\'es vectoriels de rang deux sur
des hypersurfaces de $\pp ^3$. J. f\"ur die reine und angewandte Math., Volume 490 (1997),
65-79.

\bibitem{MS1}
N. Mestrano, C. Simpson. Obstructed bundles of rank two on a quintic surface.
{\em Int. J. Math.} {\bf 22} (2011), 789-836.

\bibitem{MS2}
N. Mestrano, C. Simpson.
Seminatural bundles of rank two, degree one and $c_2=10$ on a
quintic surface. {\em Kyoto J. Math.} {\bf 53} (2013), 155-195.

\bibitem{MS3}
N. Mestrano, C. Simpson.
Moduli of sheaves. {\em Development of Moduli Theory (Mukai-60, Kyoto, 2013)}, S. Kondo, ed.,
{\sc Adv. Studies in Pure Math.}, Mathematical Society of Japan (to appear).

\bibitem{Mukai}
S. Mukai. Symplectic structure of the moduli space of sheaves on an abelian or K3
surface. Inventiones Math. 77 (1984), 101-116.

\bibitem{Nijsse}
P. Nijsse. The irreducibility of the moduli space of stable vector bundles of rank $2$
on a quintic in $\pp ^3$. Preprint arXiv:alg-geom/9503012 (1995).

\bibitem{OGradyIrred}
K. O'Grady. The irreducible components of moduli spaces of vector bundle on surfaces.
Inventiones Math., Volume 112 (1993), 585-613.

\bibitem{OGradyBasic}
K. O'Grady. Moduli of vector bundles on projective surfaces: some basic results.
Inventiones Math., Volume
123 (1996), 141-207.

\bibitem{Reider}
I. Reider. Vector bundles of rank $2$ and linear systems on algebraic surfaces.
Ann. of Math. 127 (1988), 309-316.

\bibitem{SawantHC}
A. Sawant. Hartshorne's connectedness theorem. Preprint, Tata Institute (2011)

\bibitem{YoshiokaAppli}
K. Yoshioka. An application of exceptional bundles to the moduli of stable sheaves on a K3 surface.
Arxiv preprint  alg-geom/9705027 (1997).

\bibitem{YoshiokaK3}
K. Yoshioka.
Irreducibility of moduli spaces of vector bundles on K3 surfaces.
Arxiv preprint math/9907001 (1999).

\bibitem{YoshiokaAbelian}
K. Yoshioka. Moduli spaces of stable sheaves on abelian surfaces.
Math. Ann. 321 (2001), 817-884.

\bibitem{Zuo}
K. Zuo. Generic smoothness of the moduli spaces of rank two stable vector bundles
over algebraic surfaces. {\em Math. Z.} {\bf 207} (1991), 629-643.


\end{thebibliography}

\end{document}